\documentclass[a4paper, 11pt, draft]{amsart}
\usepackage{amsmath, amsthm, amscd, amssymb, amsfonts, amsxtra, amssymb, latexsym, bm}
\usepackage{enumerate}
\usepackage{verbatim}
\usepackage{float}
\usepackage{inputenc, t1enc}

\newcommand{\Z}{\mathbb{Z}}
\newcommand{\N}{\mathbb{N}}
\newcommand{\ff}{\mathbb{F}}
\newcommand{\Tr}{\operatorname{Tr}}

\newcommand{\CC}{\mathcal{C}}

\newcommand{\G}{\Gamma}

\newtheorem{thm}{Theorem}[section]
\newtheorem{prop}[thm]{Proposition}

\newtheorem{coro}[thm]{Corollary}

\theoremstyle{definition}
\newtheorem{rem}[thm]{Remark}
\newtheorem{exam}[thm]{Example}
\newtheorem{defi}[thm]{Definition}

\theoremstyle{remark}

\usepackage{color}

\begin{document}
\numberwithin{equation}{section}
\title[Spectra of GP-graphs and irreducible cyclic codes ]{Spectral properties of generalized Paley graphs \\ 
and their associated irreducible cyclic codes}
\author{Ricardo A.\@ Podest\'a, Denis E.\@ Videla}
\dedicatory{\today}
\keywords{Irreducible cyclic codes, generalized Paley graphs, spectra, cartesian decomposable}
\thanks{2010 {\it Mathematics Subject Classification.} 
Primary 94B15, 05C25;\, Secondary 05C50, 11P05}
\thanks{Partially supported by CONICET, FonCyT and SECyT-UNC}
\address{Ricardo A.\@ Podest\'a. FaMAF -- CIEM (CONICET), Universidad Nacional de C\'ordoba. 
	\newline Av.\@ Medina Allende 2144, Ciudad Universitaria, (5000), C\'ordoba, Argentina. \newline
{\it E-mail: podesta@famaf.unc.edu.ar}}
\address{Denis E.\@ Videla. FaMAF -- CIEM (CONICET), Universidad Nacional de C\'ordoba. 
\newline	Av.\@ Medina Allende 2144, Ciudad Universitaria, (5000), C\'ordoba, Argentina.
	\newline {\it E-mail: dvidela@famaf.unc.edu.ar}}

\begin{abstract}
For $q=p^m$ with $p$ prime and $k\mid q-1$, we consider the generalized Paley graph $\Gamma(k,q) = Cay(\ff_q, R_k)$, with 
$R_k=\{ x^k : x \in \ff_q^* \}$, and the irreducible $p$-ary cyclic code 
$\CC(k,q) = \{(\Tr_{q/p}(\gamma \omega^{ik})_{i=0}^{n-1})\}_{\gamma \in \ff_q}$, with $\omega$ a primitive element of $\ff_q$ and 
$n=\tfrac{q-1}{k}$.
We first express the spectra of $\Gamma(k,q)$ in terms of Gaussian periods. Then, we show that the spectra of $\Gamma(k,q)$ and $\CC(k,q)$ are mutually determined by each other if further $k\mid \tfrac{q-1}{p-1}$. 
We give $Spec(\Gamma(k,q))$ explicitly for those graphs associated with irreducible 2-weight cyclic codes in the semiprimitive and exceptional cases.
We also compute $Spec(\Gamma(3,q))$ and $Spec(\Gamma(4,q))$.
\end{abstract}

\maketitle

\section{Introduction}
The connection between cyclic codes and graphs was first noticed almost 50 years ago. 
Several different relations can be found in the literature. 
In the first place, some authors constructed 2-weight irreducible cyclic codes from strongly regular graphs (srg) and conversely. 
In fact, Delsarte (1972, \cite{D}) constructed strongly regular graphs from 2-weight irreducible cyclic codes and established a direct relationship between the distance matrix of the code and the adjacency matrix of the graphs. 
Also, van Lint and Schrijver (1982, \cite{vLSch}) gave another construction of strongly regular graphs from cosets of mulplicative subgroups of $\ff_q$. Moreover, Calderbank and Kantor (1987, \cite{CK}) showed that any projective 2-weight irreducible cyclic code can be obtained from a strongly regular graph. 
Later, Haemers, Peeters and van Rijckevorsel (1999, \cite{HPvR}) constructed binary linear codes from the row-spam of the adjacency matrix of a given regular graph. In particular, they showed that the code corresponding to the Paley graph is the quadratic residue code. 

From the early 2000's on, some authors constructed linear codes with good decoding properties (PD-sets) from the row-spam of the incidence matrix (2003, \cite{GK}) and from adjacency matrix of Paley graphs (2004, \cite{KL}). This result was later extended by considering generalized Paley graphs (2013, \cite{SL}). 
Finally, in 2013, the spectrum of cyclic codes with many arbitrary number of zeros was computed using the spectrum of Hermitian form graphs (\cite{LHFG}, \cite{ZZDX}). By using quadratic forms, Zhou et al.\@ (\cite{ZZDX}) extend the computations to other families of codes. 

In this paper we will establish a spectral relation between generalized Paley graphs $\Gamma(k,q)$ and certain irreducible cyclic codes $\CC(k,q)$ and deduce some structural properties for the graphs from them.

\subsubsection*{Generalized Paley graphs}
If $G$ is a group and $S$ is a subset of $G$ not containing $0$, the associated Cayley graph $\Gamma = X(G,S)$ is the digraph with 
vertex set $G$ and where two vertices $u,v$ form a directed edge from $u$ to $v$ in $\Gamma$ if and only if $v-u \in S$. 
If $S$ is symmetric ($S=-S$), then $X(G,S)$ is a simple (undirected) graph. 

The \textit{generalized Paley graph} is the Cayley graph (\textit{GP-graph} for short)
\begin{equation} \label{Gammas}
\G(k,q) = X(\ff_{q},R_{k}) \quad \text{with } \quad R_{k} = \{ x^{k} : x \in \ff_{q}^*\}.
\end{equation} 
That is, $\G(k,q)$ is the graph whit vertex set $\ff_{q}$ and two vertices $u,v \in \ff_{q}$ are neighbors (directed edge) 
if and only if 
$v-u=x^k$ for some $x\in \ff_q^*$. 
Notice that if $\omega$ is a primitive element of $\ff_{q}$, then $R_{k} = \langle \omega^{k} \rangle = \langle \omega^{(k,q-1)} \rangle$. This implies that 
$$\G(k,q)= \G((k,q-1),q)$$
and that it is a $\frac{q-1}{(k,q-1)}$-regular graph. 
Thus, we will assume that $k \mid q-1$. The graphs $\Gamma(k,q)$ are denoted $GP(q,\frac{q-1}k)$ in \cite{LP}.
The graph $\G(k,q)$ is simple if $q$ is even or if $k \mid \tfrac{q-1}2$ for $p$ odd, and it is 
connected if $\tfrac{q-1}{k}$ is a primitive divisor of $q-1$ (i.e.\@ $\frac{p^m-1}{k}$ does not divide $p^{a}-1$ for any $a<m$).
For $k=1, 2$ we get the complete graph $\G(1,q)=K_q$ and the classic Paley graph $\Gamma(2,q) = P(q)$. 

We will also consider the complementary graph $\bar \G(k,q)=X(\ff_{q^m}, R_k^c \smallsetminus \{0\})$.
Let $\alpha$ be a primitive element of $\ff_{q^m}$ and consider the cosets $R_k^{(j)}=\alpha^j R_k$ for $0 \le j \le n$. 
If we put $\G^{(j)}(k,q) = \G(\ff_{q^m}, R_k^{(j)})$ then we have the disjoint union 
$$\bar \G(k,q) = \G^{(1)}(k,q) \cup \cdots \cup \G^{(n-1)}(k,q)$$
where $\G^{(j)}(k,q) \simeq \G(k,q)$ for every $1\le j \le n-1$ 
(same proof as in Lemma 4.2 in \cite{PV1}).

The spectrum of a graph $\G$, denoted $Spec(\G)$, is the spectrum of its adjacency matrix $A$ (i.e.\@  
the set of eigenvalues of $A$ counted with multiplicities).
If $\Gamma$ has different eigenvalues $\lambda_0, \ldots, \lambda_t$ with multiplicities $m_0,\ldots,m_t$, we write 
as usual 
$$Spec(\Gamma) = \{[\lambda_0]^{m_0}, \ldots, [\lambda_t]^{m_t}\}.$$
It is well-known that an $n$-regular graph $\G$ has $n$ as one of its eigenvalues, with multiplicity equal to the number of connected components of $\G$. 
There are few cases of known spectrum of GP-graphs. For instance unitary Cayley graphs over rings $X(R,R^*)$, where $R$ is a finite abelian ring and $R^*$ is the group of units (see \cite{Ak}, this includes the cases $X(\Z_n,\Z_n^*)$) and $X(\ff_{q^m},S_\ell)$ with $S_\ell=\{ x^{q^\ell+1} : x\in \ff_{q^m}^*\}$ where $\ell \mid m$ (see \cite{PV1}, this includes the classical Paley graphs $P(q)$).
We will compute $Spec(\G(k,q))$, which includes $X(\ff_{q^m},S_\ell)$ since $S_\ell=R_{q^\ell+1}$.

\subsubsection*{Irreducible cyclic codes}
A linear code of length $n$ over $\ff_q$ is a vector subspace $\CC$ of $\ff_q^n$. The weight of a codeword $c=(c_{0},\ldots,c_{n-1})$ is the number $w(c)$ of its nonzero coordinates. 
The spectrum of $\CC$, denoted $Spec(\CC) = (A_0,\ldots,A_{n})$, is the sequence of frequencies 
$A_{i} = \#\{c\in\mathcal{C} : w(c)=i\}$.
A linear code $\CC$ is cyclic if for every $(c_{0},\ldots,c_{n-1})$ in $\CC$ the shifted codeword  
$(c_{1},\ldots,c_{n-1},c_{0})$ is also in $\CC$.  
An important subfamily of cyclic codes is given by the irreducible cyclic codes. 
For $k \mid q-1$ we will be concerned with the weight distribution of the $p$-ary irreducible cyclic codes 
\begin{equation} \label{Ckqs}
\mathcal{C}(k,q) = \big \{ c_{\gamma} = \big( \Tr_{q/p}(\gamma\, \omega^{ki}) \big)_{i=0}^{n-1} : \gamma \in \ff_{q} \big \} 
\end{equation}
where $\omega$ is a primitive element of $\ff_{q}$ over $\ff_p$.
These are the codes with zero $\omega^{-k}$ and length
\begin{equation} \label{N and n} 
n = \tfrac {q-1}N \qquad \text{with} \qquad N = \gcd(\tfrac{q-1}{p-1}, k).
\end{equation}

The computation of the spectrum of (irreducible) cyclic codes is in general a difficult task. 
There are several papers on the computation of the spectra of some of these codes using exponential sums. 
Baumert and McEliece were one of the first authors to compute the spectrum in terms of Gauss sums (\cite{BMc}, \cite{Mc}). 
In 2009, Ding showed (\cite{Di1}, \cite{Di2}) that the weights of irreducible cyclic codes can be calculated in terms of Gaussian periods
\begin{equation} \label{gaussian}
\eta_{i}^{(N,q)} = \sum_{x\in C_{i}^{(N,q)}} \zeta_p^{\Tr_{q/p}(x)} \in \mathbb{C}, \qquad 0 \le i \le N-1,
\end{equation} 
where $\zeta_p = e^{\frac{2\pi i}{p}}$ and 
$C_{i}^{(N,q)} = \omega^{i} \, \langle \omega^N \rangle$
is the coset in $\ff_q$ of the subgroup $\langle \omega^{N} \rangle$ of $\ff_{q}^*$. 
From Theorem 14 in \cite{DY}, we have the following integrality results:
\begin{equation} \label{int gp}
\eta_i^{(N,q)} \in \Z \qquad \text{and} \qquad N \eta_i^{(N,q)} +1 \equiv 0 \pmod p.
\end{equation}

The spectra of irreducible cyclic codes with few weights are known.  
In 2002, Schmidt and White conjectured a characterization of 2-weight irreducible cyclic codes into three cases: 
subfield subcodes, semiprimitive and exceptional. 
In 2007, Wolfmann and Vega characterized all 1-weight irreducible cyclic codes (\cite{VW}) and found a 
characterization of all projective 2-weights codes (not necessarily irreducible). 
We now list several cases where the spectra of irreducible cyclic codes is known. 
The survey of Ding and Yang (\cite{DY}) summarizes all the results on spectra of irreducible cyclic codes until 2013. 
If $q=p^m$, $s=q^t$ and $n$, $N$ as in \eqref{N and n}, then we have:
\begin{itemize}
\item Conditions on $k$ and $n$:
\begin{enumerate}[$(a)$]
	\item the \textit{semiprimitive case}: $k\mid q^j+1$ and $j \mid \frac t2$ (\cite{BMc}, \cite{DG}, \cite{Mc});
	\item $n$ is a prime power (\cite{SB});
	\item $k$ is a prime with $k\equiv 3 \pmod 4$ and $\mathrm{ord}_{q}(k)=\frac{k-1}{2}$ (\cite{BMy}). 	
\end{enumerate} 
\item Small values of $N$: $k\mid s-1$, arbitrary $q$ and 
\begin{enumerate}[$(a)$]
	\item $N=1$ (\cite{Di1}, \cite{DY});
	\item $N=2$ (\cite{BMc}, \cite{DY});
	\item $N=3$, with $p\equiv 1 \pmod 3$ (\cite{Di1}, \cite{Di2}, \cite{DY}) or 
	$p\equiv 2 \pmod 3$ and $mt$ even (\cite{DY});
	\item $N=4$ and $p\equiv 1 \pmod 4$ (\cite{Di1}, \cite{Di2}, \cite{DY}).
\end{enumerate} 
\item Few weights:
\begin{enumerate}[$(a)$]
  \item 1-weight codes (\cite{VW}, \cite{DY});
  \item 2-weight codes (\cite{SW}, \cite{DY});
  \item some 3-weight codes (\cite{Di2}).
\end{enumerate} 
\end{itemize}
\vspace{-.5em}

\subsubsection*{Outline and results}
We now give a brief summary of the results of the paper. 
In Section~2 we compute the spectrum of $\Gamma(k,q)$ and 
$\bar \Gamma(k,q)$ in terms of Gaussian periods
(Theorem~\ref{Spectro Gkq}).
If further $k\mid \frac{q-1}{p-1}$, both such spectra are integral. 
In the next two sections we consider the case of $(k,q)$ being a semiprimitive pair (see Definition 3.1).  
In Theorem \ref{semiprimitive} we deduce explicit expressions for the spectra of the semiprimitive graphs $\G(k,q)$ and 
$\bar \Gamma(k,q)$. 
From this, we obtain that they are strongly regular graphs (srg); so, we give their srg-parameters and their intersection arrays as distance regular graphs, and we also prove that they are Latin square graphs in half of the cases (Proposition \ref{srg}). 
In Section~4 we study the property of being Ramanujan. 
In Proposition \ref{rama car}, we give a partial characterization of all 
semiprimitive pairs $(k,q)$ such that $\G(k,q)$ is Ramanujan, showing that $k$ can only take the values $2, 3, 4$ or $5$.

In Section 5, we show that the spectrum of the graph $\Gamma(k,q)$ is closely related 
with the corresponding one for the code $\CC(k,q)$. 
In Theorem \ref{pesoaut}, we show that the eigenvalue $\lambda_\gamma$ of $\G(k,q)$ and the weight $c_\gamma$ the code $\CC(k,q)$ 
satisfy the simple expression 
$$\lambda_\gamma = \tfrac{q-1}k - \tfrac{p}{p-1} \, w(c_\gamma).$$
As a consequence, in Corollary~\ref{Ram} we get a lower bound for the minimum distance $d$ of $\CC(k,q)$ in the case that $\G(k,q)$ 
is Ramanujan. In the following section, using this result and the known weight distribution of the codes $\CC(3,q)$ and $\CC(4,q)$, we compute the spectra of $\G(3,q)$ and $\G(4,q)$.  
Finally, the exceptional case is treated in the last section, where we compute the spectrum of 
$\G(k,q)$ and $\CC(k,q)$ for the eleven exceptional pairs (Theorem \ref{exceptional}).

\section{The spectrum of generalized Paley graphs and Gaussian periods}
Here, we compute the spectrum of $\Gamma(k,q)$ and of its complement $\bar \G(k,q)=X(\ff_q, R_k^c \smallsetminus \{0\})$, in terms of Gaussian periods.
Let $n=\frac{q-1}k$ and $\eta_0 = \eta_0^{(N,q)}, \ldots, \eta_{k-1} = \eta_{k-1}^{(N,q)}$ 
be the Gaussian periods as in \eqref{gaussian}. 
Also, let $\eta_{i_1},\ldots, \eta_{i_s}$ denote the different Gaussian periods not equal to $n$ and, for $0 \le i \le k-1$, define the following numbers  
\begin{equation} \label{numbers} 
\mu = \#\{0 \le i \le k-1 : \eta_i =n\} \ge 0 \quad \text{and} \quad 
\mu_{i} = \#\{ 0 \le j \le k-1 : \eta_{j} = \eta_{i}\} \ge 1.
\end{equation}

We now show that, under mild conditions, both GP-graphs and their complements are what we call \textit{$GP$-spectral}, that is 
their spectra are determined by Gaussian periods.

\begin{thm} \label{Spectro Gkq}
Let $q=p^m$ with $p$ prime and $k \in \N$ such that $k\mid q-1$ and also $k \mid \frac{q-1}2$ if $p$ is odd.
If we put $n=\frac{q-1}k$ then, in the previous notations, we have 
\begin{equation} \label{spec Gkq} 
Spec(\G(k,q)) = \{ [n]^{1+\mu n}, [\eta_{i_1}]^{\mu_{i_1} n}, \ldots, [\eta_{i_s}]^{\mu_{i_s} n} \}
\end{equation}
and $Spec(\bar \G(k,q)) = \{ [(k-1)n]^{1+\mu n}, [-1-\eta_{i_1}]^{\mu_{i_1} n}, \ldots, [-1-\eta_{i_s}]^{\mu_{i_s} n} \}$,
where $\eta_i^{(k,q)}$ are the Gaussian periods as in \eqref{gaussian}. 
Moreover, we have:
\begin{enumerate}[($a$)]
\item $\G(k,q)$ and $\bar \G(k,q)$ are connected if and only if $\mu=0$. 

\item If $k \mid \frac{q-1}{p-1}$ then $Spec(\G(k,q))$ and $Spec(\bar \G(k,q))$ are integral.
\end{enumerate}
\end{thm}
 
\begin{proof}
We first compute the eigenvalues of $\G(k,q)$. It is well-known that the spectrum of a Cayley graph $X(G,S)$ is determined by the irreducible characters of $G$. If $G$ is abelian, each irreducible character $\chi$ of $G$ induces an eigenvalue of 
$X(G,S)$ by the expression 
\begin{equation}\label{Eigencayley}
\chi(S) = \sum_{g \in S} \chi(g)
\end{equation} 
with eigenvector $v_{\chi} = \big( \chi(g) \big)_{g\in G}$.

For $\Gamma(k,q)$ we have $G=\ff_q$ and $S=R_k$. 
The irreducible characters of $\ff_{q}$ are $\{ \chi_\gamma \}_{\gamma \in \ff_q}$ where 
\begin{equation}\label{chi gamma y}
\chi_{\gamma}(y) = \zeta_{p}^{\Tr_{q/p}(\gamma y)} 
\end{equation}
for $y \in \ff_q$. Thus, since $R_{k} = \langle \omega^{k} \rangle = C_{0}^{k,q}$, the eigenvalues of $\G(k,q)$ are 
\begin{equation} \label{chi gamma y2}
\lambda_{\gamma} = \chi_{\gamma}(R_{k}) = 
    \sum_{y\in R_{k}}\chi_{\gamma}(y)=\sum_{y\in C_{0}^{(k,q)}}\zeta_{p}^{\Tr_{q/p}(\gamma y)}.
\end{equation}

We have $\ff_q = \{0\} \cup C_0^{(k,q)} \cup \cdots \cup C_{k-1}^{(k,q)}$, a disjoint union, and  
$\# C_i^{(k,q)} = \# \langle \omega^k \rangle = \frac{q-1}k$ for every $i = 0, \ldots, k-1$. 
For $\gamma=0$ we have 
$$\lambda_0=\chi_0(R_k) = |R_k| = n,$$
since $\chi_0$ is the principal character. 
This is in accordance with the fact that since $\G(k,q)$ is $n$-regular with $n=\frac{q-1}k$, then $n$ is an eigenvalue of $\G(k,q)$. 
If $\gamma \in C_{i}^{(k,q)}$ then $\gamma y$ runs over $C_{i}^{(k,q)}$ when $y$ runs over $C_{0}^{(k,q)}$
and thus, by \eqref{chi gamma y2}, we have 
\begin{equation}\label{AutTq}
\lambda_{\gamma} = \sum_{x\in C_{i}^{(k,q)}} \zeta_{p}^{\Tr_{q/p}(x)} = \eta_{i}^{(k,q)}
\end{equation} 
which does not depend on $\gamma$.

Let $\eta_{i_1}, \ldots,\eta_{i_s}$ be the different Gaussian periods.
Notice that each $\gamma \in C_{i_l}^{(k,q)}$ gives the same $\lambda_\gamma$ and $|C_{i_l}^{(k,q)}|=|C_0^{(k,q)}|=n$.
Thus, it is clear that the multiplicity of $\lambda_\gamma$ is 
$$m(\lambda_0)  = 1+\sum_{\substack{0\le j \le k-1 \\[.5mm] \eta_j=n}} |C_j^{(k,q)}|  \qquad \text{and} \qquad 
m(\lambda_\gamma) = \sum_{\substack{ 0\le j \le k-1 \\[0.5mm] \eta_{i_l} = \eta_j}} |C_j^{(k,q)}| \qquad (\text{for } \gamma \ne 0),$$
that is $m(n)=1+\mu n$ and $m(\eta_{i_l}) = \mu_{i_l} n$  for $1 \le l \le s$.

If $A$ is the adjacency matrix of $\G(k,q)$ then $J-A-I$ is the adjacency matrix of $\bar \G(k,q)$, where $J$ stands for the all 
$1$'s matrix. Since $\G(k,q)$ is $n$-regular with $q$ vertices, then $\bar \G(k,q)$ is $(q-n-1)$-regular, that is $\bar \lambda_0= q-n-1=(k-1)n$.
The remaining eigenvalues of $\bar \G(k,q)$ are $-1-\lambda$ where $\lambda$ are the non-trivial eigenvalues, and 
hence the result follows by  \eqref{spec Gkq}.

It remains to show items ($a$) and ($b$).

\noindent
($a$) Being $n$-regular, $\G(k,q)$ is connected if and only if the multiplicity of $n$ is 1, i.e.\@ if $\mu=0$. 

\noindent
($b$) Expression \eqref{spec Gkq} gives the spectra of $\G(k,q)$ and $\bar \G(k,q)$ in terms of the Gaussian periods $\eta_{i}^{(N,q)}$ for $k \mid \frac{q-1}2$. If $k$ satisfies $k \mid \frac{q-1}{p-1}$ then $k=N$, by \eqref{N and n}, and hence all the Gaussian periods 
$\eta_i^{(k,q)}$ are integers, by \eqref{int gp}, which clearly implies that $Spec(\G(k,q))$ is integral. The same happens for 
$Spec( \bar \G(k,q))$, and the result follows.  
\end{proof}

\begin{rem}
If all the Gaussian periods are different, $\eta_i \ne \eta_j$ for $0 \le i < j \le k-1$, then 
$Spec(\G(k,q)) = \{ [n]^1, [\eta_0]^n, [\eta_1]^n, \ldots, [\eta_{k-1}]^n\}$. 
This holds, for instance, for Paley graphs $P(q)=\G(2,q)$, as one can see in \eqref{spec paley} in Example \ref{spec G2q}, 
and also for $\G(3,q)$ and $\G(4,q)$ in the non-semiprimitive case (see Theorems \ref{gp3q} and \ref{gp4q} and ($iii$) in 
Remark \ref{rem gp34q}). 
\end{rem}

The \textit{period polynomial} is defined by 
$\Psi_{(k,q)}(X) = \prod_{i=0}^{k-1} (X-\eta_i^{(k,q)})$
where $\eta_i^{(k,q)}$ are the Gaussian periods. In the previous notations we have the following direct consequence of 
Theorem \ref{Spectro Gkq}.

\begin{coro} \label{period pol}
Let $q=p^m$ with $p$ prime and let $k\in \N$ such that $k \mid q-1$ and $k\mid \frac{q-1}2$ if $p$ is odd. 
Then, the period polynomial satisfies
$$\Psi_{(k,q)}(X) = \frac{P_{\G(k,q)}(X)}{X-n}$$ 
where $n=\frac{q-1}k$ and $P_{\G(k,q)}(X)$ denotes the characteristic polynomial of the graph $\G(k,q)$.
\end{coro}

The Ihara zeta function $\zeta_\G(u)$ for a regular graph $\G$ has a determinantial expression in spectral terms. 
For a GP-graph $\G(k,q)$, we have (see (8.3) in \cite{PV1} for details)
\begin{equation} \label{ihara}
\zeta_{\G} (u) = \frac{(1-u^2)^{q-\frac{nq}2}}{\prod_{i=1}^q (1-\lambda_i u - (n-1)u^2)^{m_i}}
\end{equation}
where $\{ [\lambda_1]^{m_1}, \ldots, [\lambda_q]^{m_q}\}$ is the spectrum of $\G(k,q)$.

In the next example we will obtain the already known spectrum of classical Paley graphs $P(q)$ throughout Gaussian periods.

\begin{exam}[Paley graphs] \label{spec G2q}
We will use that the periods for $k=2$ are known and that $\G=\G(2,q)=P(q)$ for $q$ odd.
Let $q=p^{m}$ with $p$ an odd prime and $m=2t$ and let $n=\frac{q-1}2$. 
We will show that $\Gamma(2,q)$ is a connected strongly regular graph having integral spectrum 
$$Spec(\Gamma(2,q)) = \{ [n]^1, [\eta]^{n}, [-1-\eta]^n\}$$
with $\eta = \eta_0^{(2,q)} = \tfrac{-1-\sigma p^t}2$ 
where $\sigma = 1$ if $p\equiv 1 \pmod 4$ and $\sigma = (-1)^t$ if $p\equiv 3 \pmod 4$. 

In fact, $2\mid \frac{q-1}2$ since $p$ is odd and $m$ is even, and hence 
$Spec(\Gamma) = \{ [n]^{1+\mu n}, [\eta_0]^{\mu_0 n}, [\eta_1]^{\mu _1 n} \}$ 
by Theorem \ref{Spectro Gkq}, where $\eta_i=\eta_i^{(2,q)}$ are the Gaussian periods for $i=1,2$. The spectrum is integral since $2$ also divides $\frac{p^m-1}{p-1}$, $m$ being even. 

The above periods are given in Lemma~12 in \cite{DY}. In our notations, we have $\eta_1=-1-\eta_0$ and 
$\eta_0= \frac{-1+(-1)^{m-1}\sqrt q}2$ if $p\equiv 1 \pmod 4$ and 
$\eta_0= \frac{-1+(-1)^{m-1} \sqrt{-1}^m \sqrt q}2$ if $p\equiv 3 \pmod 4$. Using that $m=2t$ we get that 
$$\eta_0 = \begin{cases} 
\frac{-1-p^t}2 & \qquad \text{if $p\equiv 1 \pmod 4$}, \\[1mm]
\tfrac{-1-(-1)^t p^t}2 & \qquad \text{if $p\equiv 3 \pmod 4$}.
\end{cases} $$
Now, it is clear that $\eta_0 \ne \eta_1$ and $\eta_i \ne n$, $i=0,1$, and hence $\mu=0$ and $\mu_0=\mu_1=1$. In this way, the spectrum is as given above. This coincides with the known spectrum 
\begin{equation} \label{spec paley}
Spec(P(q)) = \{[\tfrac{q-1}2]^1, [\tfrac{\sqrt{q}-1}2]^{n}, [\tfrac{-\sqrt{q}-1}2]^{n}\}.
\end{equation}
From this and using \eqref{spec paley} one can get the Ihara zeta function of $P(q)$.
The period polynomial is 
$$\Psi_{(2,q)}(X) = (x-\eta)^n(x+1+\eta)^n = (x^2+x-\eta(1+\eta))^n = (x^2+x+\tfrac{q-\sqrt{q}}4)^n.$$

The graph $\G(2,q)$ 
is connected since the multiplicity of the regularity degree is 1. Since $p$ is odd and $m$ is even it is clear that the eigenvalues are integers. Finally, since $\G(2,q)$ is regular and connected and has exactly 3 eigenvalues it is a strongly regular graph. 
All these facts are of course well-known.
\hfill $\lozenge$
\end{exam}

\section{The spectrum of semiprimitive GP-graphs $\G(k,q)$}
Using that the Gaussian periods in the semiprimitive case are known, we give the spectrum of the corresponding 
GP-graphs $\G(k,q)$ and of their complements $\bar \G(k,q)$. We then use these spectra to give structural properties of 
the graphs $\G(k,q)$ in this case.

We recall that, in the study of 2-weight irreducible cyclic codes, the \textit{semiprimitive case} corresponds to 
$-1$ being a power of $p$ modulo $k$ (see \cite{SW}). Since we also assumed from the beginning that $k \mid q-1$, where $q=p^m$ with $p$ prime, we have that semiprimitiveness is equivalent to $k=2$ and $q$ odd or else 
\begin{equation} \label{semip cond}
k \mid p^t+1 \quad \text{ for some } \quad t \mid \tfrac m2
\end{equation}
for $k>2$. With respect to the graphs, notice that if $k=2$ with $q$ odd, we have to further require that $q\equiv 1 \pmod{4}$, since otherwise the graph $\G(k,q)$ is directed. 
On the other hand, if $k=p^{\frac m2}+1$ then the graph $\G(k,q)$ is not connected (see Proposition 4.6 in \cite{PV1}).

\begin{defi} 
We say that $(k, q)$ with $k=2$ and $q\equiv 1 \pmod 4$ or $k>2$ satisfying \eqref{semip cond} and $k \ne p^{\frac m2}+1$ is a \textit{semiprimitive pair} of integers. 
We will refer to $\G(k,q)$ as a \textit{semiprimitive GP-graph} (hence $\G(k,q)$ is simple and connected).
\end{defi}

\begin{exam} \label{semipr}
($i$) For instance, if $p=3$ and $m=4$, to find the semiprimitive pairs $(k,81)$ we take $k\mid 3^2+1=2\cdot 5$ and $k\mid 3^1+1=4$.
Hence $k=2, 4$ or $5$, while $k=10$ is not allowed since $10=3^{\frac m2}+1$. 

\noindent ($ii$) 
Three infinite families of semiprimitive pairs, for $p$ prime and $m=2t \ge 2$, are:   

\begin{itemize} 
\item $(2,p^{2t})$ with $p$ odd,  

\item $(3,p^{2t})$ with $p\equiv 2 \pmod 3$ and $t\ge 1$ (where $t\ge 2$ if $p=2$), 

\item $(4,p^{2t})$ with $p\equiv 3 \pmod 4$ and $t\ge 1$ (where $t\ge 2$ if $p=3$). 
\end{itemize}
The first of the three families of pairs give rise to the classical Paley graphs $\G(2,p^{2t})$.

\noindent ($iii$) 
Another infinite family of semiprimitive pairs is given by $(p^\ell+1,p^m)$ 
with $p$ prime, $m\ge 2$, $\ell\mid m$ and $\frac{m}{\ell}$ even. 
They give the GP-graphs $\G(q^\ell+1,q^m)$, with $q=p$, considered in \cite{PV1} for $q$ a power of $p$.

\noindent ($iv$) 
Using the previous definition and items ($ii$) and ($iii$), we give a list of the smallest semiprimitive pairs $(k,q)$ with $q=p^m$ for $p=2,3,5,7$ and $m=2,4,6,8$.  

\begin{table}[H] \label{tabla 1}
\caption{Values of $k$ for small semiprimitive pairs $(k,p^m)$.}
\begin{tabular}{|c|p{1.25cm}|p{3cm}|p{3.25cm}|p{3.25cm}|} \hline 
      & $m=2$ & $m=4$ & $m=6$ & $m=8$ \\ \hline
$p=2$ & --    & 3 & 3 & 5 \\ \hline 
$p=3$ & --    & 2, 4, \textbf{5} & 2, 4, \textbf{7}, \textbf{14} & 2, 4, \textbf{5}, \textbf{10}, \textbf{41} \\ \hline 
$p=5$ & 2, \textbf{3}  & 2, \textbf{3}, 6, \textbf{13} & 2, \textbf{3}, 6, \textbf{7}, \textbf{9}, \textbf{14}, \textbf{18}, 
\textbf{21},  \textbf{42}, \textbf{63} & 2, \textbf{3}, 6, \textbf{13}, 26, \textbf{313} \\ \hline
$p=7$ & 2, \textbf{4}  & 2, \textbf{4}, \textbf{5}, 8, \textbf{10}, \textbf{25} & 2, \textbf{4}, \textbf{5}, 8, 
\textbf{10}, \textbf{25}, \textbf{43}, 50, \textbf{86}, \textbf{172} & 2, \textbf{4}, \textbf{5}, 8, \textbf{10}, \textbf{25}, 50, 
\textbf{1201} \\  
\hline
\end{tabular}
\end{table}

Here we have marked in bold those $k$ which are different from 2 and not of the type $p^\ell+1$ for some $p$ and $\ell$, showing that in general there are much more semiprimitive graphs $\G(k,q)$ than Paley graphs or the GP-graphs of the form $\G(p^\ell+1,p^m)$. 
\hfill $\lozenge$
\end{exam}

We now give the spectrum of semiprimitive GP-graphs $\G(k,q)$ explicitly. We will need the following notation.
Define the sign 
\begin{equation} \label{signo} 
\sigma = (-1)^{s+1}
\end{equation}
where $s=\frac{m}{2t}$ and $t$ is the least integer $j$ such that $k \mid p^j+1$ (hence $s \ge 1$).

\begin{thm} \label{semiprimitive}
Let $(k, q)$ be a semiprimitive pair with $q=p^m$, $m$ even, and put $n=\frac{q-1}k$.  
Then, the spectrum of $\G=\G(k,q)$ and $\bar \G = \bar \G(k,q)$ are respectively given by 
\begin{align*}
Spec(\G) & = \{ [n]^1, [\lambda_1]^n, [\lambda_2]^{(k-1)n}\}, \\[1.5mm] 
Spec(\bar \G) & = \{[(k-1)n]^1, [(k-1) \lambda_2]^n, [-1-\lambda_2]^{(k-1)n}\},
\end{align*}
where 
\begin{equation} \label{autoval semip}
	 \lambda_1 = \tfrac{\sigma (k-1) p^{\frac m2}-1}{k}
	   \qquad  \text{and} \qquad  
		\lambda_2 = - \tfrac{\sigma p^{\frac m2}+1}{k} 
\end{equation}
with $\sigma$ as given in \eqref{signo}. 
\end{thm}

\begin{proof}
We first compute the spectrum of $\G = \G(k,q)$, which by Theorem \ref{Spectro Gkq} is given in terms of Gaussian periods. 
From Lemma 13 in \cite{DY} the Gaussian periods $\eta_j^{(k,q)}$, for $j=0, \ldots, k-1$, are given by:

($a$) If $p$, $\alpha=\frac{p^t+1}k$ and $s$ are all odd then 
\begin{equation} \label{gp1} 
\eta_j^{(k,q)} = \begin{cases}
\frac{(k-1)\sqrt q -1}k & \qquad \text{if } j=\frac k2, \\[1mm]
-\frac{\sqrt q +1}k & \qquad \text{if } j \ne \frac k2.
\end{cases}
\end{equation}

($b$) In any other case we have $\sigma=(-1)^{s+1}$ and 
\begin{equation} \label{gp2} 
\eta_j^{(k,q)} = \begin{cases}
\frac{\sigma (k-1)\sqrt q -1}k & \qquad \text{if } j=0, \\[1mm]
-\frac{\sigma \sqrt q +1}k & \qquad \text{if } j \ne 0.
\end{cases}
\end{equation}

Thus, by Theorem \ref{Spectro Gkq}, the spectrum of $\G(k,q)$ is $Spec(\G(k,q)) = \{ [n]^1, [\eta_{k/2}]^n, [\eta_0]^{(k-1)n}\}$ if 
$p$, $\alpha$, $s$ are odd or $Spec(\G(k,q)) = \{ [n]^1, [\eta_0]^n, [\eta_1]^{(k-1)n}\}$ otherwise.

Suppose we are in case ($a$), i.e.\@ $p$, $\alpha$ and $s$ are odd. 
Then we have 
$$\lambda_1 = \eta_{k/2} =  \tfrac{(k-1)p^{\frac m2}+1}k 
\qquad  \text{and} \qquad 
\lambda_2 = \eta_j = \eta_0 = -\tfrac{p^{\frac m2}+1}{k} \quad (j\ne \tfrac k2).$$
It is clear that $\lambda_2\ne n$ and $\lambda_2\ne \lambda_1$. Also, $n \ne \lambda_1$ since 
$k \ne p^{\frac m2}+1$. 
Thus, all three eigenvalues are different and their corresponding multiplicities are as given in the statement. 

In case ($b$), we have 
$$\eta_0 = \tfrac{\sigma(k-1)p^{\frac m2}+1}k \qquad \text{and} \qquad 
\eta_j = -\tfrac{\sigma p^{\frac m2}-1}k \quad \text{ for } j\ne 0.$$ 
Again, one checks that $\eta_0 \ne \eta_j$, $\eta_0\ne n$ and $\eta_j \ne n$ for every $j\ne \frac k2$.  
Thus, the corresponding multiplicities are as stated in the proposition. 

Combining cases ($a$) and ($b$) we get \eqref{autoval semip}, as we wanted to show.
Finally, the spectrum of $\bar \G(k,q)$ follows by Theorem \ref{Spectro Gkq}.
\end{proof}

\noindent
\textit{Note.} Since $\lambda_1, \lambda_2 \in \Z$, we have that 
$\sigma=\pm 1$ if and only if $k \mid p^{\frac m2} \pm 1$, respectively. 

\begin{rem}
Recently, we computed the spectrum of the GP-graphs $\G_{q,m}(\ell)=\G(q^\ell+1,q^m)$ and $\bar \G_{q,m}(\ell)$, 
with $\ell \mid m$ and $\frac{m}{\ell}$ even (see Theorem~3.5 and Proposition~4.3 in \cite{PV1}), 
by using certain trigonometric sums associated with the quadratic forms 
$$Q_{\gamma,\ell}(x) = \Tr_{p^m/p} (\gamma x^{q^\ell+1}) \quad \text{with } \quad \gamma \in \ff_{p^m}^*.$$ 
By ($ii$) in Example \ref{semipr}, the graph $\G(p^\ell+1,p^m)$, i.e.\@ with $q=p$ prime, 
is semiprimitive and hence its spectrum is given by Theorem \ref{Spectro Gkq}. Indeed, 
$Spec(\G(p^\ell+1,p^m)=\{ [n]^1, [\lambda_1]^{n}, [\lambda_2]^{p^\ell n} \}$ 
where 
\begin{equation} \label{spec pl+1} 
n = \tfrac{p^m-1}{p^\ell+1}, \qquad \quad \lambda_1 = \tfrac{\sigma p^{\frac m2 + \ell}-1}{p^\ell+1}, \qquad \quad 
\lambda_2 = -\tfrac{\sigma p^{\frac m2}+1}{p^\ell+1},
\end{equation}
with $\sigma = (-1)^{\frac{m}{2\ell}+1}$. 
It is reassuring that both computations of the spectrum coincide after using these two different methods. The same happens for the complementary graphs. 
\end{rem}

Some graph invariants are given directly in terms of the spectrum $\lambda_0, \lambda_1, \ldots, \lambda_{q-1}$ of $\G$. For instance,  the energy of $\G$ and the number of walks of length $r$ in $\G$ are given by
$$E(\G)=\sum_{i=1}^q |\lambda_i| \qquad \text{ and } \qquad w_r(\G) = \sum_{i=1}^q \lambda_i^r.$$ 
Also, if $\G$ is connected and $n$-regular, the number of spanning trees of $\G$ is given by 
$$t(\G)= \tfrac 1q \theta_1 \cdots \theta_{q-1},$$ 
where $\theta_i = n-\lambda_i$ are the Laplace eigenvalues for $i=1,\ldots,q-1$.    
In our case, after some straightforward calculations we have 
\begin{equation} \label{energy}
\begin{split}
E(\G(k,q)) & = n \, \{ 2(p^{\frac m2} +\sigma \lambda_2) + 1 +\sigma \}, \\[1mm]
w_r(\G(k,q)) & = n \, ( n^{r-1} + \sigma^r \{(p^{\frac m2} +\sigma \lambda_2)^r + (k-1) (\sigma \lambda_2)^r \}), \\[1mm]
t(\G(k,q)) & = (-\sigma)^{q-1} \, p^{\frac m2 (q-3)} (\lambda_2+1)^n \lambda_2^{(k-1)n}. 
\end{split}
\end{equation}

Now, we will give some structural properties of the graphs $\G(k,q)$ throughout the spectrum.

\begin{prop} \label{srg}
Let $(k, q)$ be a semiprimitive pair with $q=p^m$, $m=2ts$ where $t$ is the least integer satisfying $k\mid p^t+1$ and put 
$n=\frac{q-1}k$. 
Then we have:

\begin{enumerate}[$(a)$]
\item $\G(k,q)$ and $\bar \G(k,q)$ are primitive, non-bipartite, integral, strongly regular graphs with corresponding parameters 
$srg(q,n,e,d)$ and $srg(q,(k-1)n,e',d')$ given by 
$$e=d +  (\sigma p^{\frac m2} + 2\lambda_2), \quad d=n + (p^{\frac m2} +\lambda_2) \lambda_2),  
\quad e'=q-2-2n+d, \quad d'=q-2n + e.$$

\smallskip 
\item $\G(k,q)$ and $\bar \G(k,q)$ are distance regular graphs of diameter 2 with intersection arrays 
$$\mathcal{A} = \{n,n-e-1;1,d\} \qquad \text{and} \qquad 
\bar{\mathcal{A}} = \{(k-1)q, n-d; 1, q-2n+e\}.$$

\smallskip
\item If $s$ is odd then $\G(k,q)$ and $\bar \G(k,q)$ are Latin square graphs  
with parameters
\begin{equation} \label{LSquares} 
L_\delta(w) = srg(w^2, \delta(w-1), \delta^2-3\delta+w, \delta(\delta-1)),
\end{equation}
where $w=f-g$, $\delta=-g$ and $f>0>g$ are the non-trivial eigenvalues of $\G(k,q)$ or $\bar \G(k,q)$. 
\end{enumerate}
\end{prop}

\begin{proof}
We will use the spectral information from Proposition \ref{semiprimitive}. 
We prove first the results for $\G(k,q)$.

($a$) Since the multiplicity of the degree of regularity $n$ is $1$, the graph is connected. 
Also, one can check that $-n$ is not an eigenvalue of $\G(k,q)$ and hence the graph is non-bipartite.
Now, since $k\mid p^t+1$ then $k\mid p^{ts}+1$ if $s$ is odd and $k\mid p^{ts}-1$ if $s$ is even, 
hence $\beta = \frac{p^{\frac m2}+\sigma}k$ is an integer. Hence, since 
$\lambda_1=\sigma(p^{ts}-\beta)$ and $\lambda_2=-\sigma \beta$, by \eqref{autoval semip}, the eigenvalues are all integers (we also know this from Theorem \ref{Spectro Gkq}).

Finally, since the graph is connected, $n$-regular with $q$-vertices and has exactly 3 eigenvalues, it is a strongly regular graph with parameters $srg(q,n,e,d)$. 
We now compute $e$ and $d$. It is known that the non-trivial eigenvalues of an srg graph are of the form 
$\lambda^\pm = \tfrac 12 \{ (e-d) \pm \Delta \}$ 
where $\Delta = \sqrt{(e-d)^2 + 4(n-d)}$.  
Thus, $d= n +\lambda^ + \lambda^-$ and $e = d + \lambda^+ + \lambda^-$.
From this and \eqref{autoval semip} the result follows.  

($b$) Similar to the proof of Corollary 5.3 in \cite{PV1}.

($c$) Note that the regularity degree of $\G=\G(k,q)$ equals the multiplicity of a non-trivial eigenvalue by Theorem \ref{semiprimitive}. Thus, $\G$ is of pseudo-Latin square type graph (PL), of negative Latin square type (NL) or is a conference graph 
(see Proposition 8.14 in \cite{CvL}). 
By definition, a conference graph satisfy 
$2n+(q-1)(e-d)=0$.
It is easy to check that this condition holds for $\G(k,q)$ if and only if 
$\G(1,4)=K_4$, and hence $\G$ is not a conference graph. Put $w=f-g$, where $f,g$ are the non-trivial eigenvalues with $f>0>g$, and 
$\delta=-g$. Then, 
$\G$ is a pseudo-Latin square graph with parameters 
$$PL_\delta(w) = srg(w^2, \delta(w-1), \delta^2-3\delta+w, \delta(\delta-1))$$
or a negative Latin square graph with parameters
\begin{equation} \label{NL}
NL_\delta(w) = srg(w^2, \delta(w+1), \delta^2+3\delta-w, \delta(\delta+1)).
\end{equation}
It is clear that $n=\delta(w-1)$ if and only if $s$ is odd and that for $s$ even $k \ne \delta(w+1)$. 

Thus, we are only left to prove that in the case of $s$ odd, $\G$ is actually a Latin square graph. 
Hence, it is enough to show that $\G$ is geometric, that is, it is the point graph of a partial geometry $pg(a,b,\alpha)$ with parameters
$srg((a+1)(ab+\alpha)\alpha^{-1}, (b+1)a, a-1+b(\alpha-1), (b+1)\alpha)$.
To see this, we use a result of Neumaier (Theorem 7.12 in \cite{CvL}) asserting that if $g<-1$ is integer and 
\begin{equation} \label{conditions}
f+1 \le \tfrac{1}{2} g(g+1)(d+1)
\end{equation}
then $\G$ is the point graph of a partial geometry $pg(a,b,\alpha)$ with $\alpha=b$ or $b+1$. 
We know that for $s$ odd we have that $g<-1$ is integer. Expression \eqref{conditions} takes the form
$$k-1 \le -\tfrac 12 (\sigma \lambda_2+1) (n+1+(\sigma p^{\frac m2} + \lambda_2) \lambda_2).$$
By straightforward calculations, this inequality is equivalent to 
\begin{equation}\label{ineq large}
2k^3(k-1)+k(2p^m+4p^{\frac m2}-2kp^{\frac m2}+2-2k+k^2)\le p^{\frac{3m}2}+3p^m+3p^{\frac m2}+1.
\end{equation}
Notice that the l.h.s.\@ of \eqref{ineq large} increases as $k$ does. By hypothesis $s>1$ is odd, thus 
$k\le K:=p^{\frac{m}{6}}+1$ since $s$ satisfies $k\mid p^{\frac{m}{2s}}+1$ and $s\ge 3$. Therefore, the l.h.s.\@ 
of \eqref{ineq large} is less than 
$$C_K := 2K^4 - K^3 + 2Kp^m + (2-K)2Kp^{\frac m2} +2K(1-K).$$ 
It is not difficult to show that $C_K$ is less than the right hand side of \eqref{ineq large}. Therefore, \eqref{ineq large} is true in this case, and this implies that $\G$ satisfies \eqref{conditions} as required.

Now, it is easy to see that $\bar \G(k,q)$ is also a primitive non-bipartite integral strongly regular graph 
with parameters 
and intersection array as stated. 
The proof that $\bar \G(k,q)$ is a Latin square if $s$ is odd is analogous to the previous one for $\G(k,q)$ and we omit the details. 
Since $\G(k,q)$ is a pseudo-Latin  graph $PL_\delta(w)$ then $\bar \G(k,q)$ is a pseudo-Latin graph $PL_{\delta'}(w)$ with 
$\delta'=u+1-\delta$. Thus, it is enough to check that $\bar \G(k,q)$ is geometric, that is, that inequality \eqref{conditions} holds in this case also with $g'$ in place of $g$. This completes the proof.
\end{proof}

\begin{exam} \label{tablitas srg}
From Theorem \ref{semiprimitive} and Proposition \ref{srg} we obtain Table \ref{tablita srg} below. 
Here $s=\frac{m}{2t}$ where $t$ is the least integer such that $k\mid p^t+1$.
We have marked in bold those graphs $\G(k,p^m)$ with $k \ne p^\ell+1$ for some $\ell \mid \frac m2$.
We point out that, for instance, the graphs with $q=7^4$ do not appear in the Brouwer's lists (\cite{Br}) of strongly regular graphs. \renewcommand{\arraystretch}{1.1}
\begin{small}
\begin{table}[h!]
\caption{Smallest semiprimitive graphs: srg parameters and spectra}  \label{tablita srg}
\begin{tabular}{|c|c|c|c|c|c|} \hline 
 graph 			 & srg parameters & spectrum & $t$ & $s$ & latin square \\ \hline
$\G(3,2^4)$  & $(16, 5, 0, 2)$   &  $\{ [5]^1, [1]^{10}, [-3]^{5} \}$ & 1& 2& no \\ 
$\bar\G(3,2^4)$  & $(16, 10, 6, 6)$   &  $\{ [10]^1, [2]^{5}, [-2]^{10} \}$ & 1& 2& no \\ \hline
$\G(3,2^6)$  & $(64, 21, 8, 6)$  &  $\{ [21]^1, [5]^{21}, [-3]^{42} \}$ & 1& 3& $L_3(8)$ \\ 
$\bar\G(3,2^6)$  & $(64, 42, 26, 30)$  &  $\{ [42]^1, [2]^{42}, [-6]^{21} \}$ & 1& 3& $L_6(8)$ \\ \hline
$\G(\bm{3},\bm{5^2})$ & $(25, 8, 3, 2)$   &  $\{ [8]^1, [3]^{8}, [-2]^{16} \}$ & 1& 1& $L_2(5)$ \\ 
$\bar\G(\bm{3},\bm{5^2})$ & $(25, 16, 9, 12)$   &  $\{ [16]^1, [1]^{16}, [-4]^{8} \}$ & 1& 1& $L_4(5)$ \\ \hline 
$\G(\bm{3},\bm{5^4})$  & $(625, 208, 63, 72)$ &  $\{ [208]^1, [8]^{416}, [-17]^{208} \}$ & 1& 2& no \\ 
$\bar\G(\bm{3},\bm{5^4})$  & $(625, 416, 279, 272)$ &  $\{ [416]^1, [16]^{208}, [-9]^{416} \}$ & 1& 2& no \\ \hline
$\G(4,3^4)$  & $(81, 20, 1, 6)$ &    $\{ [20]^1, [2]^{60}, [-7]^{20},  \}$ & 1& 2& no\\ 
$\bar\G(4,3^4)$  & $(81, 60, 45, 42)$ &    $\{ [60]^1, [6]^{20}, [-3]^{60},  \}$ & 1& 2& no\\ \hline
$\G(4,3^6)$  & $(729, 182, 55, 42)$ &  $\{ [182]^1, [20]^{182}, [-7]^{546} \}$ & 1& 3& $L_7(27)$ \\ 
$\bar\G(4,3^6)$  & $(729, 546, 405, 420)$ &  $\{ [546]^1, [6]^{546}, [-21]^{182} \}$ & 1& 3& $L_{21}(27)$ \\ \hline
$\G(\bm{4},\bm{7^2})$  & $(49, 12, 5, 2)$    &  $\{ [12]^1, [5]^{12}, [-2]^{36} \}$ & 1& 1& $L_2(7)$ \\ 
$\bar\G(\bm{4},\bm{7^2})$  & $(49, 36, 25, 30)$    &  $\{ [36]^1, [1]^{36}, [-6]^{12} \}$ & 1& 1& $L_6(7)$ \\ \hline
$\G(\bm{4},\bm{7^4})$  & $(2401, 600, 131, 156)$ & $\{ [600]^1, [12]^{1800}, [-37]^{600} \}$ & 1& 2& no \\ 
$\bar\G(\bm{4}, \bm{7^4})$  & $(2401, 1800, 1332, 1355)$ & $\{ [1800]^1, [36]^{600}, [-13]^{1800} \}$ & 1& 2& no \\ \hline
$\G(\bm{5},\bm{3^4})$  & $(81, 16, 7, 2)$    &  $\{ [16]^1, [7]^{16}, [-2]^{64} \}$ & 2& 1& $L_2(9)$ \\ 
$\bar\G(\bm{5},\bm{3^4})$  & $(81, 64, 49, 56)$    &  $\{ [64]^1, [1]^{64}, [-8]^{16} \}$ & 2& 1& $L_8(9)$ \\ \hline
$\G(\bm{5},\bm{7^4})$  & $(2401,480,119, 90)$ &  $\{ [480]^1, [39]^{480}, [-10]^{1920} \}$ & 2& 1& $L_{10}(49)$ \\ 
$\bar\G(\bm{5},\bm{7^4})$  & $(2401,1920,1560, 1529)$ &  $\{ [1920]^1, [9]^{1920}, [-40]^{480} \}$ & 2& 1& $L_{40}(49)$ \\ \hline
\end{tabular}
\end{table}
\end{small}
\hfill $\lozenge$
\end{exam}

\begin{rem}
($i$) Notice that if we take $h = \min\{|f|, |g|\}$, then for $s$ even (in the previous notations), $\G(k,q)$ satisfy the same parameters as in \eqref{NL} with $\delta$ replaced by $h$, that is $\G(k,q)$ is a strongly regular graph with parameters, 
in terms of the eigenvalues, given by
$$\widetilde{NL} = srg(w^2, h(w+1), h^2+3h-w, h(h+1)).$$

\noindent ($ii$) 
Since the complement of $\G(k,q)$ is a Latin square $L_\delta(w)$ and its complement $\bar \G(k,q)$ is also a Latin square 
$L_{\delta'}(w)$ with $\delta'=u+1-\delta$, by Proposition \ref{srg}, the graphs 
can be constructed from a complete set of mutually orthogonal Latin squares (MOLS). This is equivalent to the existence of an affine plane of order $w$. Since $w=f-g=p^{\frac m2}$ with $m$ even, the semiprimitive GP-graphs give a standard way to construct complete sets of MOLS of order $w$ and affine planes of order $w$, with $w$ a power of a prime.      
\end{rem}

\section{All Ramanujan semiprimitive GP-graphs} 
If $\Gamma$ is an $n$-regular graph, then $n$ is the greatest eigenvalue of $\G$. 
A connected $n$-regular graph is called \textit{Ramanujan} if 
\begin{equation} \label{ram cond}
\lambda(\G) := \max_{\lambda \in Spec(\Gamma)} \{|\lambda| : |\lambda|\ne n \} \le 2\sqrt{n-1}. 
\end{equation} 
Next, we present a complete characterization of the semiprimitive generalized Paley graphs which are Ramanujan.

\begin{thm} \label{rama car}
Let $q=p^m$ with $p$ prime and let $(k,q)$ be a semiprimitive pair.
Then, the graph $\G=\G(k,q)$ is Ramanujan if and only if

$\G$ is the classic Paley graph $\G(2,q)$, with $q \equiv 1 \pmod 4$, or
$m$ is even and $k, p, m$ are as
in one of the following cases:
\begin{enumerate}[$(a)$]
	\item $k=3$, $p=2$ and $m\ge 4$.  
	\item $k=3$, $p \ne 2$ with $p\equiv 2 \pmod 3$ and $m\ge 2$.
	\item $k=4$, $p=3$ and $m\ge 4$. 
	\item $k=4$, $p\ne 3$ with $p\equiv 3 \pmod 4$ and $m\ge 2$. 
	\item $k=5$, $p=2$ and $m\ge 8$ with $4\mid m$.
	\item $k=5$, $p\ne 2$ with $p\equiv 2,3\pmod 5$ and $m\ge 4$ with $4\mid m$.
	\item $k=5$, $p\equiv 4 \pmod 5$ and $m\ge 2$. 
\end{enumerate} 
Moreover, $\bar{\G}(k,q)$ is Ramanujan for every semiprimitive pair $(k,q)$.
\end{thm}

\begin{proof}
We begin by noticing that $k=1$ is excluded since $(1,q)$ is not a semiprimitive pair and that 
$k=2$ corresponds to the classic Paley graph $\G(2,q)$, with $q \equiv 1 \pmod 4$, 
which is well-known to be Ramanujan. 
The proof follows in three steps. 

So it is enough to consider semiprimitive pairs $(k,p^m)$ with $k>2$.

\noindent \textit{Step 1.} 
Here we prove that if $\G$ is Ramanujan with $(k,p^m)$ a semiprimitive pair with $k\ne 2$, then $3\le k\le 5$.

Note that $\lambda(\G)=|\lambda_1|$, for $k\ge 3$.
Since $\G$ is Ramanujan, 
\eqref{ram cond} reads
$$\tfrac 1k |\sigma(k-1)p^{\frac m2}-1| \le 2\sqrt{\tfrac{p^m-(k+1)}{k}}.$$ 
This inequality is equivalent to $(k-1)^2 p^m-2\sigma(k-1) p^{\frac m2}+1\le 4k(p^{m}-(k+1))$  which holds if and only if 
\begin{equation}\label{rammast2}
4k(k+1)+1\le p^m (4k-(k-1)^2)+ 2(k-1)\sigma p^{\frac m2}.
\end{equation}

Assume first that $\sigma=-1$. Then, \eqref{rammast2} takes the form
\begin{equation}\label{rammast}
2(k-1)p^{\frac m2}+4k(k+1)+1\le p^m (4k-(k-1)^2).
\end{equation} 
Since the left hand side of this inequality is positive, we have that $4k-(k-1)^2>0$, and this can only happen if $ k \le 5$.

Now, let $\sigma=1$. In this case, inequality \eqref{rammast2} is equivalent to 
\begin{equation}\label{rammast3}
0\le (4k-(k-1)^2) p^m+ 2(k-1)p^{\frac m2} -(2k+1)^2.
\end{equation}
Suppose $k>5$ and consider the polynomial 
$$P_{k}(x)=(4k-(k-1)^2) x^2+ 2(k-1)x -(2k+1)^2.$$ 
Hence, $P_k(x)$ has negative 
leading coefficient and its discriminant is given by 
$$\Delta(k)= 4 \big( (k-1)^2 - ((k-3)^2-8)(2k+1)^2 \big).$$
Since $(2k+1)^2> (k-1)^2$, the sign of $\Delta(k)$ depends on $(k-3)^2-8$. Since $k>5$, we have that $(k-3)^2-8>0$ and thus $\Delta(k)<0$. So, the quadratic polynomial $P_{k}$ have no real roots and since $P_{k}(0)=-(2k+1)^2<0$, we obtain that $P_{k}(x)<0$ for all $x\in \mathbb{R}$, in particular $P_{k}(p^{\frac m2})<0$ for all $p$ and $m$, contradicting \eqref{rammast3}. 
Therefore, if $\G$ is Ramanujan then $k\le 5$, as we wanted to show.

\noindent \textit{Step 2.} 
We now show that the pair $(k,q)$ semiprimitive with $k\le 5$ can only happen as stated in the theorem; and, in these cases, $\G(k,q)$ is Ramanujan.

As mentioned at the beginning, the case $k=1$ is excluded and $k=2$ corresponds to the classic Paley graph, which is Ramanujan.
If $k=3$, then necessarily $p\equiv 2 \pmod 3$ and $m$ is even, for if not the pair $(k,p^m)$ is not semiprimitive. 
In this case, \eqref{rammast2} is given by
$49\le 8 p^m+4\sigma p^{\frac m2}$.
The worst possibility is when $\sigma =-1$, and in this case the previous inequality reads 
$$12+\tfrac{1}{4}\le p^{\frac m2}(2p^{\frac m2}-1).$$
This clearly holds if and only if $p$ is odd and $m\ge 2$ or $p=2$ and $m\ge 4$, and thus $\G(3,p^m)$ is Ramanujan in 
these cases.

If $k=4$, then we must have $p\equiv 3 \pmod 4$ and $m$ is even, for if not the pair $(k,p^m)$ is not semiprimitive. 
In this case, \eqref{rammast2} is given by 
$ 81\le 7p^m+6\sigma p^{\frac m2}$. 
As before, the worst case is when $\sigma =-1$, and thus the inequality is equivalent to 
$$11+ \tfrac{4}{7}\le p^{\frac m2}(p^{\frac m2}-\tfrac{6}{7}) .$$
This holds if and only if $p>3$ and $m\ge 2$ or $p=3$ and $m\ge 4$ and hence $\G(4,p^m)$ is Ramanujan in these cases.

In the last case, if $k=5$ the pair $(5,p^m)$ is semiprimitive if and only if $p\equiv 2,3 \pmod 5$ and $4\mid m$ or 
else $p\equiv 4 \pmod 5$ and $m\ge 2$ even. 
On the other hand, in this case \eqref{rammast2} is given by
$121 \le 4p^m+8\sigma p^{\frac m2}$,
which is equivalent to 
$$30+\tfrac{1}{4}\le p^{\frac m2}(p^{\frac m2}+2\sigma).$$
If $p=2$, then neccesarily $m\ge 8$ since $4\mid  m$ and $m=4$ does not satisfy the above inequality.
Clearly, the inequality holds for $p\equiv 2,3 \pmod 5$ with $p\ne 2$ and $4\mid m$.
Finally, notice that the right hand side of the inequality increases when $p$ increases. 
The first prime $p$ with $p\equiv 4 \pmod 5$ is $p=19$ that crearly satisfies the inequality for $m\ge 2$ even, 
so we obtain that the inequality holds for all of primes $p\equiv 4 \pmod 5$ with $m\ge 2$ even.
In this way we have shown that $\G(5,p^m)$ is Ramanujan in all the cases in the statement.

\noindent \textit{Step 3.} 
Finally, we consider the complementary graphs $\bar{\Gamma}=\bar{\G}(k,q)$. We have that 
$$\lambda(\bar{\G}) = |\bar{\lambda}_1|=(k-1) \tfrac{p^{\frac m2}+\sigma}{k}$$ 
and the regularity degree of $\bar{\G}$ is $n(k-1)$. 
Notice that we can assume that $k>2$, since $k=2$ correspond to the classic Paley graph which is complementary, and hence Ramanujan. 
Also, without loss of generality we can assume that $\sigma=1$. 
Inequality \eqref{ram cond} becomes 
\begin{equation} \label{ram comp}
(k-1)\tfrac{p^{\frac m2}+1}{k}\le 2\sqrt{\tfrac{(p^m-1)(k-1)-k}{k}}
\end{equation}
which is equivalent to $(k-1)^{2}p^m+2(k-1)^2p^{\frac m2}+(k-1)^2 \le 4k(p^m(k-1)-(2k-1))$ and therefore we have   
$$2(k-1)^2p^{\frac m2}+(k-1)^2+4k(2k-1)\le p^{m}(4k(k-1)-(k-1)^2).$$
Notice that $(k-1)^2+4k(2k-1)=(3k-1)^2$ and $4k(k-1)-(k-1)^2=(k-1)(3k+1)$. 
Let us consider the quadratic polynomial $Q_{k}(x)=x^2-b_{k} x-c_k$, where
$$b_{k}=\tfrac{2(k-1)}{(3k+1)} \quad \text{and } \quad c_{k}=\tfrac{(3k-1)^2}{(k-1)(3k+1)}.$$
Hence, $\bar \G(k,q)$ is Ramanujan if and only \eqref{ram comp} holds, that is if and only if $Q_k(p^{\frac m2})>0$.

Clearly $b_{k}<1$ and $4c_{k}<15$, this implies that the greatest real root $r$ of $Q_k$ satisfies $$r=\tfrac{b_{k}}{2}+\tfrac{1}{2}\sqrt{b_k^2+4c_k}<\tfrac 12 +2<3.$$
Since $(k,p^m)$ is a semiprimitive pair and $k>2$, we have that $p^{\frac m2}\ge 3$. This implies that $Q_{k}(p^{\frac m2})>0$ since $Q_{k}$
has a positive leading coefficient. Therefore $\bar \G(k,p^m)$ is Ramanujan for all semiprimitive par $(k,p^m)$ with $\sigma=1.$
The case $\sigma=-1$ can be proved analogously.
\end{proof}

The previous result gives the following eight infinite families of Ramanujan 
semiprimitive GP-graphs:
\begin{enumerate}[$(a)$]
\item \quad $\{ \G(2,q)\}$ with 
 $q\equiv 1 \pmod 4$ (Paley graphs),
\item \quad $\{ \G(3,4^{t})\}_{t\ge 2}$, 
\item \quad $\{ \G(3,p^{2t})\}_{t\ge 1}$ with $p \equiv 2 \pmod 3$ and $p\ne 2$, 
\item \quad $\{ \G(4,9^{t})\}_{t\ge 2}$,
\item \quad $\{ \G(4,p^{2t})\}_{t\ge 1}$ with $p \equiv 3 \pmod 4$ and $p\ne 3$, 
\item \quad $\{ \G(5,16^{t})\}_{t\ge 2}$,
\item \quad $\{ \G(5,p^{4t})\}_{t\ge 1}$, with $p \equiv 2,3 \pmod 5$ and $p\ne 2$,
\item \quad $\{ \G(5,p^{2t})\}_{t\ge 1}$ with $p \equiv 4 \pmod 5$.
\end{enumerate}
Notice that five of them are valid for an infinite number of primes.

\begin{rem}
($i$) 
The Ramanujan GP-graphs $\G(k,q)$ with $k=p^\ell+1$ are characterized in Theorem~8.1 in \cite{PV1}. 
There, we proved that $\Gamma_{q,m}(\ell) = \G(p^{\ell}+1,p^m)$, with $\ell \mid m$ such that $m_\ell$ even 
and $\ell \ne \frac m2$, is Ramanujan if and only if $q=2,3,4$ with $\ell=1$ and $m\ge 4$ even. 
This says that $\G(p^{\ell}+1,p^m)$ is Ramanujan only in the cases ($b$), ($d$) and ($f$),  
giving the infinite families  
$$\{ \G(3, 4^t) \}_{t\ge 2}, \qquad \{ \G(4, 9^t)\}_{t\ge 2} \qquad \text{ and } \qquad \{\G(5, 16^t)\}_{t\ge 2}$$
of Ramanujan graphs. 
The first two families coincide with those in ($b$) and ($d$), while the third one gives just half the graphs in ($f$), precisely those with $t$ even in ($f$). 
Thus, the last proposition extends this characterization of Ramanujan GP-graphs $\G(q^\ell+1,q^m)$ to all semiprimitive pairs $(k,p^m)$,  that is in the case $q=p$. 

($ii$) 
When $p=2$, the last proposition does not give any novelty, since the possible values of $k \in \{2,3,4,5\}$ such that $(k,2^m)$ is a semiprimitive pair reduces to $k=3, 5$,  
which corresponds to the cases $p=2$ with $\ell=1, 2$ in ($i$) above.
\end{rem}

\begin{exam}
From Proposition \ref{rama car}, the following GP-graphs are Ramanujan:
$$\begin{tabular}{|p{1cm}|p{9cm}|}
\hline 
$p=2$ & $\G(3,16)$, $\G(3,64)$, $\G(3,256)$, $\G(5, 256)$, \\ \hline 
$p=3$ & $\G(2,81)$, $\G(4,81)$, $\G(\textbf{5}$, $\textbf{81})$, 
 $\G(2,729)$, $\G(4,729)$, $\G(2,6{.}561)$, $\G(4,6{.}561)$, $\G(\textbf{5}, \textbf{6{.}561})$ \\ \hline
$p=5$ & $\G(2,25)$, $\G(\textbf{3},\textbf{25})$, $\G(2,625)$, $\G(\textbf{3},\textbf{625})$, 
 $\G(2,15{.}625)$, $\G(\textbf{3},\textbf{15{.}625})$, $\G(2,390{.}625)$, $\G(\textbf{3},\textbf{390{.}625})$ \\ \hline  
\end{tabular}$$
where $6{.}561=3^8$, $15{.}625=5^6$ and $390{.}625=5^8$. In bold those graphs $\G(k,p^m)$ with $k\ne p^\ell+1$ for some $\ell \mid \frac m2$.
\hfill $\lozenge$
\end{exam}

\section{The relation between the spectra of $\Gamma(k,q)$ and $\CC(k,q)$}
In this section we relate the spectrum of the graph $\Gamma(k,q)$ given in \eqref{Gammas} with the weight distribution of the code 
$\CC(k,q)$ defined in \eqref{Ckqs}. We will see that $Spec(\CC(k,q))$ determines $Spec(\Gamma(k,q))$ and conversely.
To this end we will have to assume further that $k \mid \frac{q-1}{p-1}$. Under this hypothesis, the numbers in \eqref{N and n} become $N=k$ and $n=\frac{q-1}{k}$.

\begin{thm} \label{pesoaut}
Let $q=p^m$ with $p$ prime, $k \in \N$ such that $k \mid \frac{q-1}{p-1}$ and put $n=\tfrac{q-1}{k}$.
Let $\G(k,q)$ and $\CC(k,q)$ be as in \eqref{Gammas} and \eqref{Ckqs}. 
Then, we have: 
\begin{enumerate}[$(a)$]
\item The eigenvalue $\lambda_{\gamma}$ of $\G(k,q)$ and the weight of $c_{\gamma} \in \CC(k,q)$ are related by the expression
\begin{equation} \label{Pesoaut}
\lambda_{\gamma} = n - \tfrac{p}{p-1} w(c_{\gamma}). 
\end{equation}

\item If $\G(k,q)$ is connected the multiplicity of $\lambda_{\gamma}$ is $A_{w(c_{\gamma})}$ for all $\gamma\in \ff_q$. In particular, 
the multiplicity of $\lambda_0 = n$ is $A_0=1$. 
\end{enumerate}
\end{thm}

\begin{proof}
$(a)$ Let $c_{\gamma} \in \CC(k,q)$. Thus, $w(c_0)=w(0)=0$, by \eqref{Ckqs}. Also, if $\gamma \in C_{i}^{(k,q)}$ 
(see \eqref{gaussian}), 
from equation (12) in \cite{DY} we have
\begin{equation} \label{weight c}
w(c_{\gamma}) = \tfrac{p-1}{pk} \big(q - 1 - k  \eta_{i}^{(k,q)} \big).
\end{equation}

By Proposition \ref{Spectro Gkq}, the eigenvalues of $\G(k,q)$ are $\lambda=n$ or 
$\lambda_{\gamma} = \eta_{i}^{(k,q)}$ if $\gamma \in C_i^{(k,q)}$.
Putting this in \eqref{weight c} we get \eqref{Pesoaut} for $\gamma \ne 0$. But \eqref{Pesoaut} also holds for $\lambda=n$ and $w=0$, as we wanted.

$(b)$ The assertion about the multiplicities of the eigenvalues $\lambda_\gamma$ with $\gamma \ne 0$ is clear. 
For $\gamma=0$, we have $c_\gamma=0$, so $\lambda_0=n$. 
Every $n$-regular graph $\G$ has $n$ as one of its eigenvalues, with multiplicity equal to the number of connected components of the graph. Therefore, since $\Gamma$ is connected, the multiplicity of $\lambda_0$ equals $A_{w(0)}=A_0=1$. 
\end{proof}

Note that from \eqref{Pesoaut}, $\G(k,q)$ is integral if and only if $\CC(k,q)$ is $(p-1)$-divisible.
Actually, it is known that if $k \mid \frac{q-1}{p-1}$ then $\CC(k,q)$ is a $(p-1)$-divisible code (\cite{SW}) which is in accordance with part ($b$) of Theorem \ref{Spectro Gkq}.

\begin{rem}
Combining Theorem \ref{pesoaut} with Theorem 24 in \cite{DY}, which gives the spectrum of $\CC(k,q)$ in the semiprimitive case, we obtain 
Theorem~\ref{semiprimitive} in another way.
\end{rem}

\begin{coro} \label{Conseq}
Let $q=p^m$ with $p$ prime, $k \in \N$ such that $k \mid \frac{q-1}{p-1}$ and $n=\tfrac{q-1}{k}$ is a primitive divisor of $q-1$.
Then, $\Gamma(k,q)$ is a strongly regular graph if and only if $\CC(k,q)$ is a 2-weight code.	
\end{coro}

\begin{proof}
We use the well-known fact that a simple connected graph is strongly regular if and only if it has $3$ different eigenvalues. Since the graph $\Gamma(k,q)$ is simple and connected by hypothesis (see the comments after \eqref{Gammas}), 
the remaining assertion follows by $(a)$ in Theorem~\ref{Pesoaut}.  	
\end{proof}

\subsubsection*{A lower bound for the minimum distance of $\CC(k,q)$}
Note that if $\Gamma=\Gamma(k,q)$ with $k$ and $q$ satisfying $k\mid \frac{q-1}{p-1}$ and $n=\frac{q-1}k$ is a primitive divisor of $q-1$ then, by \eqref{Pesoaut}, we have (see \eqref{ram cond})
\begin{equation} \label{lambdamax}
\lambda(\Gamma) = \max \big\{ n-(\tfrac{p}{p-1})d, \, (\tfrac{p}{p-1}) d' - n \big\}
\end{equation} 
where $d = w_{min}$ is the minimum distance 
and $d'=w_{max}$ is the maximum weight of $\CC(k,q)$.

We now show that if $\Gamma(k,q)$ is Ramanujan, under the non-restrictive additional assumption $d \le (\frac{p-1}{p}) n$ 
we get a lower bound for the minimum distance of the associated code $\mathcal{C}(k,q)$.

\begin{coro} \label{Ram} 
In the previous notations and under the same hypothesis of Theorem \ref{pesoaut}, if 
$\G(k,q)$ is Ramanujan and the minimum distance of $\CC(k,q)$ satisfies $d \le (\tfrac{p-1}{p})n$ then 
\begin{equation} \label{rama} 
 d \ge (\tfrac{p-1}{p})(n-2\sqrt{n-1}). 
\end{equation}
More precisely, we have:
\begin{enumerate}[(a)]
\item If $\lambda(\G)=n-(\tfrac{p}{p-1})d$, then $\G(k,q)$ is Ramanujan if and only if \eqref{rama} holds. 

\item Suppose $\lambda(\G) = (\tfrac{p}{p-1})d'-n$ and $d \le (\tfrac{p-1}{p})n$. If $\G(k,q)$ is Ramanujan then \eqref{rama} holds. 
\end{enumerate} 
\end{coro}

\begin{proof}
The first assertion is a consequence of ($a$) and ($b$) in the statement. 
If $\lambda(\G)=n-(\tfrac{p}{p-1})d$, the equivalence in ($a$) follows immediately from \eqref{ram cond}. 
To show ($b$), by hypothesis we have 
$$ 0 \le n-(\tfrac{p}{p-1})d = |n-(\tfrac{p}{p-1})d| \le \lambda(\G)\le 2\sqrt{n-1}$$
since $\G$ is Ramanujan, and the proof is thus complete.
\end{proof}

\begin{exam}
Fix $m$ and suppose we take $q=p^m$ with $p$ prime and $k=p^\ell+1$ with $1 \le \ell \le m$. Consider the graph and the code
$$\Gamma_{p,m}(\ell) = \Gamma(p^\ell+1,p^m), \qquad
\CC_{p,m}(\ell) = \CC(p^\ell+1,p^m),$$
as in \eqref{Gammas} and \eqref{Ckqs}, respectively. 
Both $\Gamma_{p,m}(\ell)$ and $\CC_{p,m}(\ell)$, and their corresponding spectra, were respectively studied in \cite{PV1} and \cite{PV2}. 
In the definition of $\CC_{p,m}(\ell)$ we need to assume that $p^\ell+1 \mid p^m-1$. 
This is equivalent to the conditions $\ell \mid m$ and $\frac{m}{\ell}$ even. But this is just the condition for $\Gamma_{p,m}(\ell)$ to be a truly generalized Paley graph, i.e.\@ not a Paley graph (see \cite{PV1}, Proposition~2.5).
Also, take $\ell \ne \tfrac m2$ since for $\ell=\tfrac m2$ the graph is not connected.

Taking $q=p$ in Theorem 4.1 and Table 5 in \cite{PV2} we know that $\CC_\ell$ is a 2-weight code with weights and frequencies given by 
\begin{equation} 
\begin{aligned} \label{pesos}
w_1 &=  \tfrac{(p-1)(p^{m-1} + \varepsilon_\ell \, p^{\frac m2 +\ell-1})}{p^\ell+1} \,,  \qquad &  A_{w_1} & = n,  \\[1mm]
w_2 &=  \tfrac{(p-1)(p^{m-1} - \varepsilon_\ell \, p^{\frac m2 -1})}{p^\ell+1} \,,       \qquad &  A_{w_2} & = n p^\ell,
\end{aligned}
\end{equation}
where $n=\frac{p^m-1}{p^\ell+1}$ and $\varepsilon_\ell =(-1)^{\frac{m}{2\ell}}$. 
On the other hand, for $\ell \mid m$ with $\frac m\ell$ even, 
we have that $\G_\ell$ has three eigenvalues as given in \eqref{spec pl+1} 
--see also Theorem~3.5 in \cite{PV1} with $q=p$--.
From \eqref{spec pl+1} and \eqref{pesos}, 
it is easy to check that both $(a)$ and $(b)$ in Theorem~\ref{pesoaut} hold. In particular, we have 
$$w_1 = (k_\ell-\mu_\ell)\tfrac{(p-1)}{p} \qquad \text{and} \qquad 
w_2 = (k_\ell - \nu_\ell)\tfrac{(p-1)}{p}.$$ 
Note that this implies that $k_\ell \equiv \nu_\ell \equiv \mu_\ell \pmod p$ since $w_1, w_2$ are integers, a fact that we do not know from \cite{PV1}.

From Theorem 8.1 in \cite{PV1}, the graphs $\Gamma_{q,m}(\ell)$ are Ramanujan if and only if $\ell=1$, $q \in  \{2,3,4\}$ and $m=2t$ with $t\ge 2$. Thus, by Remark 2.10 in \cite{PV1}, we have that $\Gamma_{2,2t}(1)$ and $\Gamma_{2,4t}(2)$ are Ramanujan graphs for 
every $t\ge 2$.
From Theorem 4.1 in \cite{PV2}, the parameters of $\CC_{p,m}(1)$ and $\CC_{p,m}(2)$ are 
$$n = \frac{2^{2jt}-1}{2^j+1}, \qquad k=m, \qquad 
d= \begin{cases} 
2^{jt-1} \frac{2^{jt}-1}{2^j+1} & \quad \text{if $t$ is even}, \\[2mm] 
2^{jt-1} \frac{2^{jt}-2j}{2^j+1} & \quad \text{if $t$ is odd}, 
\end{cases}$$ 
for $j=1,2$. From here one can check that \eqref{rama} in Theorem \ref{Ram} holds. 
\hfill $\lozenge$
\end{exam}

\begin{rem}\label{semip weight}
Using Theorem \ref{pesoaut}, one can recover the weights of the irreducible semiprimitive codes (given in \cite{SW}) from the spectrum  
--obtained in Proposition \ref{semiprimitive}-- of the generalized Paley graphs $\G(k,q)$ where $(k,q)$ is a semiprimitive pair 
(i.e.\@ $k \mid p^{t}+1$ for some $t \ne \frac m2$ such that $t \mid m$ and $m_{t} = \frac{m}{t}$ even). 
If $t$ is minimal satisfying these conditions we have that   
$\sigma = (-1)^{\frac 12 m_{t}+1}$ (see \eqref{signo}). 
In fact, the nonzero weights of $\CC(k,q)$ are  
\begin{equation}\label{twoweight}
	w_1 = \tfrac{(p-1)p^{\frac m2 -1}}{m}(p^{\frac m2} - \sigma (k-1)) \qquad \text{and} \qquad 
	w_2 = \tfrac{(p-1)p^{\frac m2 -1}}{k}(p^{\frac m2} + \sigma)
	\end{equation}   
with multiplicities $n$ and $(k-1)n$, respectively.
Notice that $\sigma = -\varepsilon_{t}$, where $t=\ell$ and $\varepsilon_\ell$ as in the previous example. 
\end{rem}

\section{The spectrum of the graphs $\G(3,q)$ and $\G(4,q)$} 
We have given the spectrum of $\G(k,q)$ for $(k,q)$ a semiprimitive pair. Without this assumption, we know the spectrum in the cases $k=1, 2$
since $\G(1,q)=K_q$ and $\G(2,q^2)=P(q^2)$ are the complete and the classical Paley graphs. 
Now, using the main result of the previous section and the known spectra of irreducible cyclic codes $\CC(3,q)$ and $\CC(4,q)$, we compute the spectra of the associated GP-graphs $\G(3,q)$ and $\G(4,q)$, under certain conditions on $q$.
These graphs are known to be one of the first explicit examples of graphs satisfying $P(m,n,k)$ adjacency conditions (see \cite{An}).

\begin{thm} \label{gp3q}
Let $q=p^{m}$ with $p$ prime such that $3\mid \frac{q-1}{p-1}$ and $q\ge 5$. 
Thus, the graph $\G(3,q)$ is connected with integral spectrum given as follows:
\begin{enumerate}[$(a)$]
\item If $p\equiv 1 \pmod 3$ then $m=3t$ for some $t\in \N$ and 
$$Spec(\G(3,q)) = \big\{ [n]^1, \big[\tfrac{a\sqrt[3]{q}-1}{3}\big]^n, \big[\tfrac{-\frac{1}{2} (a+9b)\sqrt[3]{q}-1}{3}\big]^n, 
\big[\tfrac{-\frac 12 (a-9b) \sqrt[3]{q}-1}{3}\big]^n \big\} $$
where $a,b$ are integers uniquely determined by 
$$ 4\sqrt[3]{q}=a^2+27b^2, \quad a\equiv 1 \pmod 3 \quad \text{and} \quad (a,p)=1.$$

\item If $p\equiv 2 \pmod 3$ then $m=2t$ for some $t\in \N$ and 
$$Spec(\G(3,q)) = \begin{cases} 
\big\{ [n]^1, \big[\tfrac{\sqrt{q}-1}{3}\big]^{2n}, \big[\tfrac{-2\sqrt{q}-1}{3}\big]^n \big\} & 
\qquad \text{for $m\equiv 0 \pmod 4$}, \\[2mm]
\big\{ [n]^1, \big[\tfrac{2\sqrt{q}-1}{3}\big]^{n}, \big[\tfrac{-\sqrt{q}-1}{3}\big]^{2n} \big\} & 
\qquad \text{for $m\equiv 2 \pmod 4$}.
\end{cases}$$
In particular, $\G(3,q)$ is a strongly regular graph. 
\end{enumerate}
\end{thm}

\begin{proof}
Let $q=p^m$. First note that condition $3\mid \frac{q-1}{p-1}=p^{m-1}+\cdots +p+1$ implies that 
$m=3t$ if $p\equiv 1 \pmod 3$ and $m=2t$ if $p\equiv 2 \pmod 3$. We will apply Theorem \ref{pesoaut} to the code $\CC(3,q)$. The spectrum of $\CC(3,q)$ is given in Theorems 19 and 20 in  
\cite{DY}, with different notations ($r$ for our $q$, $N$ for our $k$, etc).

By $(a)$ in \eqref{Pesoaut}, the eigenvalues of $\G(3,q)$ are given by 
\begin{equation} \label{eigen3} 
\lambda_i = \tfrac{q-1}3 - \tfrac{p}{p-1} w_i
\end{equation} 
where $w_i$ are the weights of $\CC(3,q)$. 

If $p\equiv 1 \pmod 3$, by Theorem 19 in \cite{DY}, the four weights of $\CC(3,q)$ are $w_0=0$,
\begin{equation} \label{sp1} 
w_1= \tfrac{(p-1)(q-a\sqrt[3]{q})}{3p}, \qquad w_2= \tfrac{(p-1)(q+\frac 12 (a+9b) \sqrt[3]{q})}{3p}, 
\qquad w_3= \tfrac{(p-1)(q+\frac 12 (a-9b) \sqrt[3]{q})}{3p},
\end{equation}
with frequencies $A_0=1$ and $A_1=A_2=A_3=\frac{q-1}3$; 
where $a$ and $b$ are the only integers satisfying $4\sqrt[3]{q}=a^2+27b^2$, $a\equiv 1 \pmod 3$ and $(a,p)=1$.

On the other hand, if $p\equiv 2 \pmod 3$, by Theorem 20 in \cite{DY}, the three weights of $\CC(3,q)$ are 
\begin{equation} \label{sp2} 
w_0=0, \qquad w_1= \tfrac{(p-1)(q-\sqrt{q})}{3p}, \qquad w_2= \tfrac{(p-1)(q+2\sqrt{q})}{3p},
\end{equation}
with frequencies $A_0=1$, $A_1=\frac{2(q-1)}3$ and $A_2=\tfrac{q-1}3$ if $m\equiv 0 \pmod 4$ and 
\begin{equation} \label{sp3} 
w_0=0, \qquad w_1= \tfrac{(p-1)(q-2\sqrt{q})}{3p}, \qquad w_2= \tfrac{(p-1)(q+\sqrt{q})}{3p},
\end{equation} 
with frequencies $A_0=1$, $A_1=\frac{(q-1)}3$ and $A_2=\tfrac{2(q-1)}3$ if $m\equiv 2 \pmod 4$. 
By introducing \eqref{sp1}, \eqref{sp2} and \eqref{sp3} in \eqref{eigen3}, we get the eigenvalues in ($a$) and $(b$) of the statement. 

\smallskip 
To compute the multiplicities, we use $(b)$ of Theorem \ref{pesoaut}. The hypothesis that $\G(3,q)$ be connected is equivalent to  
the fact that $n=\frac{q-1}3$ is a primitive divisor of $q-1$ (see Introduction, after \eqref{Gammas}). 
We now show that this is always the case for $q\ge 5$. 

Suppose that $p\ge 5$. Then, we have that $p^{m-1}-1 \le \frac{q-1}3$, since this inequality is equivalent to $3p^{m-1}-3 \le p^m-1$ 
which holds for $p\ge 5$. This implies that $n$ is greater that $p^a-1$ for all $a<m$ and, hence, $n$ is a primitive divisor of $q-1$.
Now, if $p=2$ we only have to check that $n$ does not divide $2^{m-1}-1$, since $n> 2^{m-2}-1$. Notice that $n\mid 2^{m-1}-1$ if and only if $2^{m}-1\mid 3(2^{m-1}-1)$, which can only happen if $m=2$. If $m>2$ we have 
$3(2^{m-1}-1) \equiv 2^{m-1}-2 \not \equiv 0 \pmod 2^m-1$
as we wanted. The prime $p=3$ is excluded by hypothesis.

Thus, by part $(b)$ of Theorem \ref{pesoaut}, the multiplicities of the eigenvalues of $\G(3,q)$ are the frequencies of the weights of $\CC(3,q)$, and we are done.
\end{proof}

\begin{exam}
Let $p=7$ and $m=3$, hence $q=7^3=343$. Since $p\equiv 1 \pmod 3$, we have to find integers  $a,b$ such that
$28=a^2+27b^2$, $a\equiv 1 \pmod 3$ and $(a,7)=1$. Clearly $a=b=1$ satisfy these conditions. By ($i$) in Theorem~\ref{gp3q} we have  
$Spec(\G(3,7^3)) = \big\{ [114]^1, [9]^{114}, [2]^{114}, [-12]^{114} \big\}$.
\hfill $\lozenge$
\end{exam}

\begin{thm} \label{gp4q}
Let $q=p^m$ with $p$ prime such that $4\mid \frac{q-1}{p-1} $ and $q\ge 5$ with $q\ne 9$.  
Thus, the graph $\G(4,q)$ is integral and connected and its spectrum is given as follows:
\begin{enumerate}[$(a)$]
\item If $p\equiv 1 \pmod 4$ then $m=4t$ for some $t\in \N$ and
$$Spec(\G(4,q)) = \big\{ [n]^1, \big[\tfrac{\sqrt{q} + 4d\sqrt[4]{q}-1}{4}\big]^n, \big[\tfrac{\sqrt{q} - 4d\sqrt[4]{q}-1}{4}\big]^n, 
\big[\tfrac{-\sqrt{q} + 2c \sqrt[4]{q}-1}{4}\big]^n, \big[\tfrac{-\sqrt{q} - 2c \sqrt[4]{q}-1}{4}\big]^n \big\}$$
where $c,d$ are integers uniquely determined by 
$$ \sqrt{q}=c^2+4d^2, \quad c\equiv 1 \pmod 4 \quad \text{and} \quad (c,p)=1.$$

\item If $p\equiv 3 \pmod 4$ then $m=2t$ for some $t\in \N$ and 
$$Spec(\G(4,q)) = \begin{cases} 
\big\{ [n]^1, \big[\tfrac{\sqrt{q}-1}{4}\big]^{3n}, \big[-\tfrac{(3\sqrt{q}+1)}{4}\big]^n \big \} & 
\qquad \text{for $m\equiv 0 \pmod 4$}, \\[2mm]
\big\{ [n]^1, \big[\tfrac{3\sqrt{q}-1}{4}\big]^{n}, \big[\tfrac{-\sqrt{q}-1}{4}\big]^{3n} \big\} & 
\qquad \text{for $m\equiv 2 \pmod 4$}.
\end{cases}$$
In particular, $\G(4,q)$ is a strongly regular graph. 
\end{enumerate}
\end{thm}

\begin{proof}
The proof is similar to the one of Theorem \ref{gp3q}. We apply Theorem \ref{pesoaut} to the code $\CC(4,q)$ since the spectrum of this code is given in Theorem 21 in \cite{DY}. Thus, we leave out the details and only show that if $q\ge 5$ with $q\ne 9$ then $\frac{q-1}4$ is a primitive divisor of $q-1$ and hence $\G(4,q)$ is connected. 

Suppose that $p\ge 5$. Then, we have that $p^{m-1}-1 \le \frac{q-1}4$ since this inequality is equivalent to $4p^{m-1}-4 \le p^m-1$ 
which is true because $p\ge 5$. This implies that $n$ is greater that $p^a-1$ for all $a<m$ and hence $n$ is a primitive divisor of $q-1$.
Now, if $p=3$ we only have to check that $n$ does not divide $3^{m-1}-1$, since $n> 3^{m-2}-1$. Notice that $n\mid 3^{m-1}-1$ 
if and only if $3^{m}-1\mid 4(3^{m-1}-1)$, which can only happen if $m=2$. If $m>2$ in this case 
$4(3^{m-1}-1) \equiv 3^{m-1}-3 \not\equiv 0 \pmod 3^m-1$
as we wanted. The prime $p=2$ is excluded by hypothesis.
\end{proof}

\begin{exam}
Let $q=5^4=625$, that is $p=5$ and $m=4$. Since $p\equiv 1 \pmod 4$, we have to find integers  $c,d$ such that
$25=c^2+4d^2$, $c\equiv 1 \pmod 4$ and $(c,5)=1$. One can check that $(c,d)=(-3,2)$ satisfy these conditions and hence by ($i$) in Theorem~\ref{gp4q}, the spectrum of $\G(4,625)$ is given by 
$Spec(\G(4,625)) = \big\{ [156]^1, [16]^{156}, [1]^{156}, [-4]^{156}, [-14]^{156} \big\}$.
\hfill $\lozenge$
\end{exam}

We end this section with some remarks.
\begin{rem} \label{rem gp34q}
($i$) The integers $b$ and $d$ in Theorems \ref{gp3q} and \ref{gp4q} are determined up to sign. However, by symmetry, 
the eigenvalues in these theorems are not affected by these choices of sign.  

\noindent
($ii$) The spectrum of $\CC(k,q)$ with $k=3t$ is computed in Theorems 19 and 20 in \cite{DY}, since it requires $N=(\frac{q-1}{p-1},k)=3$.
Similarly, the spectrum of $\CC(4t,q)$ is provided by Theorems 21 in \cite{DY}. 
However, to apply Theorem \ref{pesoaut} to compute the spectrum of $\G(k,q)$ with $k=3t$ (resp.\@ $k=4t$), one needs the extra assumption 
$k\mid \frac{q-1}{p-1}$, which forces $N=k=3$ (resp.\@ $4$). 
Hence, another method must be used to compute the spectra of $\G(3t,q)$ and $\G(4t,q)$ for $t>1$. 

\noindent
($iii$) In the cases $(b)$ in Theorems \ref{gp3q} and \ref{gp4q}, the graphs $\G(3,q)$ and $\G(4,q)$ are semiprimitive 
by $(ii)$ in Example \ref{semipr}. Thus, their spectra is already given by Theorem \ref{semiprimitive} (without assuming $n$ to be a primitive divisor of $q-1$). However, the cases in $(a)$ are new.  
\end{rem}

\section{The exceptional case}
In this final section we compute the spectrum of the graphs associated with the exceptional 2-weight irreducible cyclic codes. 

Schmidt and White gave the following list of pairs $(k,q)$, with $k$ in ascending order, 
\begin{gather} \label{sporadic}
\begin{aligned}
(11,3^5), \quad (19, 5^9), \quad (35,3^{12}), \quad (37, 7^9) \quad (43, 11^7), \quad (67, 17^{33}), \\
(107,3^{53}), \quad (133, 5^{18}), \quad (163,41^{81}), \quad (323, 3^{144}), \quad (499, 5^{249}), 
\end{aligned}
\end{gather}
such that $\CC(k,q)$ is an exceptional 2-weight irreducible cyclic code, and they conjectured that these are all such codes (\cite{SW}). We refer to them as \textit{exceptional pairs}.

For all these pairs $(k,q)$, the number $n=\frac{q-1}k$ is always a primitive divisor of $q-1$.
For instance, for $(11,3^5)$, we see that $n=\frac{3^5-1}{11}=22$ is a primitive divisor of $3^5-1=242$ since $22\nmid 3^a-1$ for 
$a=1,2,3,4$. In fact, if $n$ is not a primitive divisor of $q-1$, then the code $\CC(k,q)$ would be a subfield subcode, which is not the case. Therefore, all these graphs $\G(k,q)$ are connected.

Now, we give the spectra of $\CC(k,q)$ and $\G(k,q)$ for the exceptional pairs given above. 

\begin{thm} \label{exceptional}
	Consider the eleven exceptional pairs $(k,q)$ from \eqref{sporadic}.
	The spectra of $\G(k,q)$ and $\CC(k,q)$ can be explicitly computed. 
	In Tables 3--6 we give these spectra, as well as the parameters as strongly regular graphs, for the first eight such pairs. 
\end{thm}

\begin{proof}
	From Corollary 3.2 and Table 1 in \cite{SW}, the nonzero weights of $\CC(k,q)$ are given by 
	\begin{equation} \label{weights sporadic} 
	w_1 = \tfrac 1k (p-1)p^{\theta -1}(p^{m-\theta}-\epsilon t) \qquad \text{and} \qquad w_2 = w_1 + \epsilon (p-1)p^{\theta-1},
	\end{equation}
	where, in our notations ($u=k$, $f=m$, $k=t$),
	\begin{equation} \label{tablita}
	\renewcommand{\arraystretch}{1}
	\begin{tabular}{|c|c|c|c|c|c|}
	\hline 
	$k$ & $p$ & $m$ & $\theta$ & $t$ & $\epsilon$ \\ \hline
	11 & 3  & 5  & 2  & 5  & 1 \\
	19 & 5  & 9  & 4  & 9  & 1 \\
	35 & 3  & 12 & 5  & 17 & 1 \\
	37 & 7  & 9  & 4  & 9  & 1 \\
	43 & 11 & 7  & 3  & 21 & 1 \\
	67 & 17 & 33 & 16 & 33 & 1 \\
	\hline
	\end{tabular}
	\qquad \quad 
	\renewcommand{\arraystretch}{1.15}
	\begin{tabular}{|r|r|r|r|r|r|}
	\hline 
	$k$ & $p$ & $m$ & $\theta$ & $t$ & $\epsilon$ \\ \hline
	107& 3 & 53& 25 & 53&1 \\
	133& 5 & 18&  8 & 33&$-1$\\
	163& 41& 81& 40 & 81&$1$ \\
	323& 3 &144& 70 &161&1 \\
	499& 5 &249& 123&249&1 \\
	\hline
	\end{tabular}
	\end{equation}
	From this, by Theorem \ref{pesoaut}, putting $n=\frac{q-1}{k}$, the eigenvalues of $\G(k,q)$ are 
	\begin{equation} \label{eigen sporadic}
	\lambda_1 = n-\tfrac{p}{p-1} w_1 \qquad \text{and} \qquad \lambda_2 = n-\tfrac{p}{p-1} w_2.
	\end{equation}

	To compute the multiplicities, notice that since $\G(k,q)$ is connected and $\CC(k,q)$ is a 2-weight code, 
	then $\G(k,q)$ is an strongly regular graph with parameters $srg(q,n,e,d)$.
	It is known that the eigenvalues and their multiplicities of such a graph are respectively given by
	\begin{equation} \label{multip sporadic} 
	\lambda^{\pm} =\tfrac{(e-d) \pm \Delta}2 \qquad \text{and} \qquad 
	m(\lambda^\pm)= \tfrac 12 \Big\{ (q-1) \mp \tfrac{2n + (q-1)(n-d)}{\Delta} \Big\}
	\end{equation}
	where 
	$$\Delta= \sqrt{(e-d)^2 + 4(n-d)}.$$
	Thus $\lambda_1 = \lambda^+$ and $\lambda_2 = \lambda^-$. Since 
	$\lambda_1+\lambda_2 = e-d$ and $\lambda_1 - \lambda_2=\Delta$, we obtain 
	\begin{equation} \label{ed sporadic} 
	d = n - \tfrac{(\lambda_1 - \lambda_2)^2-(\lambda_1+\lambda_2)^2}4 \qquad \text{and} \qquad 
	e=d+\lambda_1+\lambda_2,
	\end{equation}
	and hence the multiplicities can be calculated from \eqref{multip sporadic}. Since $\G(k,q)$ is connected, the multiplicities $m_i$ are also the frequencies $A_{w_i}$ of the weights $w_i$, $i=1,2$, by Theorem \ref{pesoaut}.
	
	By computing \eqref{weights sporadic}, using 
	\eqref{tablita} and performing all the calculations in \eqref{eigen sporadic}, \eqref{multip sporadic} and \eqref{ed sporadic}, one can get the spectrum of $\G(k,q)$ and $\CC(k,q)$ for all the 
	exceptional pairs $(k,q)$.
	The spectra of the first eight pairs are given in Tables 4--6.
	The remaining pairs $(163, 41^{81})$, $(323,3^{144})$ and $(499,5^{249})$ are quite unmanageable, 
	and we do not give them for readability. 
\end{proof}

\renewcommand{\arraystretch}{1.35}
\begin{table}[h!] \label{table4}
	\caption{Spectra of $\G(k,q)$ for the first 5 exceptional pairs.} \
	
	\begin{tabular}{|c|c|c|}
		\hline
		$(k,q)$ & parameters $srg(q,n,e,d)$ & $Spec(\G(k,q)) \smallsetminus \{[n]^1\}$ \\
		\hline
		$(11,3^5)$ 		& $srg(243, 22, 1, 2)$ & $\{[4]^{132}, [-5]^{110}\}$  \\ \hline
		$(19,5^9)$ 		& $srg(1{.}953{.}125, 102{.}796, 5{.}379, 5{.}412)$ & $\{[296]^{1{.}027{.}960}, [-329]^{925{.}164}\}$ \\ \hline
		$(35,3^{12})$ & $srg(531{.}441, 15{.}184, 427, 434)$ & $\{[118]^{273{.}312}, [-125]^{258{.}128}\}$ \\ \hline
		$(37,7^9)$ 		& $srg(40{.}353{.}607, 1{.}090{.}638, 282{.}771, 29{.}510)$ & $\{[584]^{30{.}537{.}864}, [-1817]^{9{.}815{.}742}\}$ \\ \hline
		$(43,11^7)$ 	& $srg(19{.}487{.}171, 453{.}190, 10{.}509, 10{.}540)$ & $\{[650]^{9{.}970{.}180}, [-681]^{9{.}516{.}990}\}$ \\ \hline
	\end{tabular}
\end{table}

The frequencies $A_{w_i}$ of the weights $w_i$ of $\CC(k,q)$ in the table below are the multiplicities $m_i$ of the corresponding eigenvalues 
$\lambda_i$ of $\G(k,q)$.

\renewcommand{\arraystretch}{1.25}
\begin{table}[h!] \label{table4b}
	\caption{Spectra of $\CC(k,q)$ for the first 5 exceptional pairs} \
	
	\begin{tabular}{|c|c|c|c|c|c|}
		\hline
		weights & $\CC(11,3^5)$ & $\CC(19,5^9)$ & $\CC(35,3^{12})$ & $\CC(37,7^9)$ & $\CC(43,11^7)$ \\ \hline
		$w_1$ 	& 22 						& 82{.}000 			& 10{.}044      	 & 934{.}332     & 411{.}400 \\
		$w_2$ 	& 18 						& 82{.}500 			& 10{.}026 				 & 936{.}390     & 412{.}610 \\ \hline
	\end{tabular}
\end{table}

The pairs $(67, 17^{33})$, $(107, 3^{53})$ and $(133, 5^{18})$ have intermediate complexity and are given separately in Tables 5 
and 6.

\renewcommand{\arraystretch}{1.255}
\begin{table}[h!] \label{table5}
	\caption{Spectra of $\G(k,q)$ and $\CC(k,q)$ for the 6th exceptional pair.} \
	\begin{tabular}{l}
		Pair $(67, 17^{33})$ \\
		\hline
		$q= 40{.}254{.}497{.}110{.}927{.}943{.}179{.}349{.}807{.}054{.}456{.}171{.}205{.}137$ \\
		$n= 600{.}813{.}389{.}715{.}342{.}435{.}512{.}683{.}687{.}379{.}942{.}853{.}808$ \\ 
		$e = 8{.}967{.}364{.}025{.}602{.}125{.}902{.}458{.}937{.}044{.}032{.}559{.}119$ \\ 
		$d = 8{.}967{.}364{.}025{.}602{.}125{.}903{.}185{.}223{.}489{.}938{.}034{.}768$ \\ \hline
		$\lambda_1 =  23{.}967{.}452{.}714{.}880{.}696{.}416$ \\
		$m_1 = 20{.}427{.}655{.}250{.}321{.}642{.}807{.}431{.}245{.}370{.}918{.}057{.}029{.}472$  \\ \hline
		$\lambda_2 = -24{.}693{.}739{.}160{.}786{.}172{.}065$ \\
		$m_2= 19{.}826{.}841{.}860{.}606{.}300{.}371{.}918{.}561{.}683{.}538{.}114{.}175{.}664$ \\ \hline
		$w_1 = 565{.}471{.}425{.}614{.}439{.}939{.}283{.}497{.}632{.}625{.}940{.}854{.}016$ \\
		$w_2 = 565{.}471{.}425{.}614{.}439{.}939{.}329{.}296{.}401{.}450{.}097{.}906{.}704$ \\ \hline
	\end{tabular}

\end{table}
\begin{table}[h!] \label{table6}
	\caption{Spectra of $\G(k,q)$ and $\CC(k,q)$ for the 7th and 8th exceptional pairs.}
	\begin{tabular}{l}
		Pair $(107, 3^{53})$ \\
		\hline
		$q = 19{.}383{.}245{.}667{.}680{.}019{.}896{.}796{.}723$ \\ 
		$n = 181{.}151{.}828{.}669{.}906{.}728{.}007{.}446$ \\ 
		$e = 360{.}610{.}649{.}595{.}226{.}895{.}872{.}817$ \\ 
		$d = 360{.}610{.}649{.}595{.}234{.}814{.}457{.}952$ \\ \hline
		
		$\lambda_1 =  419{.}685{.}012{.}154$ \\
		$m_1 = 9{.}782{.}198{.}748{.}174{.}963{.}312{.}402{.}084$  \\ \hline
		$\lambda_2 = -427{.}603{.}597{.}289$ \\
		$m_2= 9{.}601{.}046{.}919{.}505{.}056{.}584{.}394{.}638$ \\ \hline
		$w_1 = 120{.}767{.}885{.}779{.}658{.}028{.}663{.}528$ \\
		$w_2 = 120{.}767{.}885{.}780{.}222{.}887{.}736{.}490$ \\ \hline
	\end{tabular} \qquad \qquad 
	\begin{tabular}{l}
		Pair $(133, 5^{18})$ \\
		\hline
		$q = 3{.}814{.}697{.}265{.}625$ \\ 
		$n = 28{.}681{.}934{.}328$ \\ 
		$e = 215{.}848{.}943$ \\ 
		$d = 215{.}652{.}162$ \\ \hline
		$\lambda_1 = -96{.}922$ \\
		$m_1 = 2{.}868{.}193{.}432{.}800$  \\ \hline
		$\lambda_2 = 293{.}703$ \\
		$m_2= 946{.}503{.}832{.}824$ \\ \hline
		$w_1 = 22{.}945{.}625{.}000$ \\
		$w_2 = 22{.}945{.}312{.}500$ \\ \hline
	\end{tabular}
\end{table}

\begin{rem}
	In \cite{SW}, Schmidt and White gave an expression for the weights of the (conjecturally) all 2-weight irreducible cyclic codes. 
	From Theorems \ref{semiprimitive}, \ref{pesoaut} and \ref{exceptional}, we provide the frequencies of these weights in the semiprimitive and exceptional cases, thus completing the computation of the weight distributions of the connected 2-weight irreducible cyclic codes which are non subfield subcodes; i.e.\@ those associated with connected GP-graphs. 
\end{rem}


\begin{thebibliography}{XXX}
	
\bibitem{Ak} 
\textsc{R.\@ Akhtar, T.\@ Jackson-Henderson, R.\@ Karpman, M.\@ Boggess, I.\@ Jim\'enez, A.\@ Kinzel, D.\@ Pritikin}.
\textit{On the unitary Cayley graph of a finite ring}. 
Electron.\@ J.\@ Comb.\@ 16:1 (2009), RP 117, 13 pp. 

\bibitem{An} 
\textsc{W. Ananchuen}. 
\textit{On the adjacency properties of generalized Paley graphs}. 
Austral.\@ J.\@ Combin.\@ 24 (2001), 129--147.

\bibitem{BMc}{\sc L.D.\@ Baumert, R.J.\@ McEliece}. 
\textit{Weights of irreducible cyclic codes.} 
Information and Control 20 (1972), 158--175. 

\bibitem{BMy}{\sc L.D.\@ Baumert, J.\@ Mykkeltveit}.
\textit{Weight distributions of some irreducible cyclic codes.} 
DSN Progr.\@ Rep.\@ 16 (1973), 128--131.

\bibitem{Br}{\sc A.E.\@ Brouwer}.
Strongly regular graphs' page \verb|www.win.tue.nl/~aeb/graphs/srg/srgtab.html|

\bibitem{CK}{\sc R.\@ Calderbank, W.\@ Kantor}.
\textit{The geometry of two-weight codes.} 
Bull. London Math. Soc. 18 (1986) 97--122.

\bibitem{CvL} \textsc{P.J.\@ Cameron, J.H.\@ van Lint}. \textit{Designs, graphs, codes and their links}. 
Cambridge University Press, LMSST 22, 1991.


\bibitem{D}{\sc P.\@ Delsarte}.
\textit{Weights of linear codes and strongly regular normed spaces.} 
Discrete Math.\@ 3 (1972) 47--64.


\bibitem{DG}{\sc P.\@ Delsarte, J.M.\@ Goethals}.
\textit{Irreducible binary cyclic codes of even dimension.} 
Proc. Second Chapel Hill Conf. on Combinatorial Mathematics and its
Applications, Univ.\@ North Carolina, Chapel Hill, NC, (1970) 100--113.

\bibitem{Di1}{\sc C.\@ Ding}. 
\textit{The weight distribution of some irreducible cyclic codes.} 
IEEE Trans.\@ Inform.\@ Theory 55:3 (2009), 955--960.

\bibitem{Di2}{\sc C.\@ Ding}. 
\textit{A class of three-weight and four-weight codes}. 
C.\@ Xing, et al. (Eds.), Proc.\@ of the Second International Workshop on Coding Theory and
Cryptography. Lecture Notes in Computer Science, vol.\@ 5557, Springer Verlag, (2009) 34--42.

\bibitem{DY}{\sc C.\@ Ding, J.\@ Yang}. 
\textit{Hamming weights in irreducible cyclic codes}.
Discrete Math.\@  313:4 (2013), 434--446.

\bibitem{GK}{\sc  D.\@ Ghinelli, J.D.\@ Key,}. 
\textit{Codes from incidence matrices and line graphs of Paley graphs}.
Adv.\@ Math.\@ Comm.\@ 5 (2011), 93--108.

\bibitem{HPvR}{\sc  W.\@ Haemers, R.\@ Peeters, J.\@ van Rijckevorsel}. 
\textit{Binary codes of strongly regular graphs}.
Design Code. Cryptogr.\@ 17 (1999), 187--209.

\bibitem{KL}{\sc J.D.\@ Key, J.\@ Limbupasiriporn}. 
\textit{Partial permutation decoding for codes from Paley graphs}.
Cong.\@ Numer.\@ 170 (2004), 143--155.

\bibitem{LHFG}{\sc S.\@ Li, S.\@ Hu, T.\@ Feng, G.\@ Ge}. 
\textit{The weight distribution of a class of cyclic codes related to Hermitian forms graphs}. 
IEEE Trans.\@ Inform.\@ Theory 59:5 (2013), 3064--3067.

\bibitem{LP}{\sc T.K.\@ Lim, C.\@ Praeger}. 
\textit{On Generalised Paley Graphs and their automorphism groups}.
Michigan Math.\@ J.\@ 58 (2009), 294--308.

\bibitem{Mc}{\sc R.J.\@ McEliece}. 
\textit{Irreducible cyclic codes and Gauss sums}.
Combinatorics in: Proc. NATO Advanced Study Inst., Breukelen, 1974. Math. Centre Tracts 55, Math. Centrum, Amsterdam, 1974, 179--196.

\bibitem{PV1}\textsc{R.A.\@ Podest\'a, D.E.\@ Videla}.
\textit{The spectra of generalized Paley graphs and applications}.
arXiv:1812.03332, (2018).	

\bibitem{PV2}\textsc{R.A.\@ Podest\'a, D.E.\@ Videla}.
\textit{Weight distribution of cyclic codes defined by quadratic forms and related curves}.
arXiv:1903.01838, (2019).

\bibitem{SL}{\sc P.\@ Seneviratne, J.\@ Limbupasiriporn}. 
\textit{Permutation decoding from generalized Paley graphs}. 
Appl.\@ Algebra in Eng.\@ Comm.\@ and Computing 24 (2013) 225--236.

\bibitem{SB}{\sc A.\@ Sharma, G.K.\@ Bakshi}. 
\textit{The weight distribution of some irreducible cyclic codes.} 
Finite Fields Appl.\@ 18:1 (2012), 144--159.

\bibitem{SW}
\textsc{B.\@ Schmidt, C.\@ White}.
\textit{All two weight irreducible cyclic codes?}
Finite Fields Appl.\@ 8 (2002), 1--17.

\bibitem{vLSch}
\textsc{J.\@ H.\@ van Lint, A.\@ Schrijver}.
\textit{Construction of strongly regular graphs, two-weight codes and partial geometries by finite fields}. 
Combinatorica 1:1 (1981), 63--73. 

\bibitem{VW}{\sc G.\@ Vega, J.\@ Wolfmann}. 
\textit{New classes of 2-weight cyclic codes}. 
Des.\@ Codes Cryptogr.\@ 42 (2007), 327--334.

\bibitem{ZZDX} {\sc  Z.\@ Zhou, A.\@ Zhang, C.\@ Ding, M.\@ Xiong}. 
\textit{The weight enumerator of three families of cyclic codes}. 
IEEE Trans.\@ Inform.\@ Theory 59:9 (2013),  6002--6009.
\end{thebibliography}
\end{document}